\documentclass[11pt]{amsart}
\usepackage[utf8]{inputenc}

\usepackage{amssymb,amsmath,amsfonts,amsthm,array,comment,mathrsfs,times,graphicx}

\makeindex
\usepackage{geometry}
\usepackage{longtable,tabularx}
\usepackage{amsmath}
\usepackage[all,cmtip]{xy}
\usepackage{tikz-cd}
\usepackage{url}
\usepackage[normalem]{ulem}
\usepackage[hidelinks]{hyperref}
\usepackage{scalerel}

\hypersetup{bookmarksdepth=2}

\usepackage{enumitem}
\setlist[enumerate]{label=(\roman*)}

\newtheorem{theorem}{Theorem}[section]
\newtheorem{coro}[theorem]{Corollary}
\newtheorem{lemm}[theorem]{Lemma}
\newtheorem{prop}[theorem]{Proposition}

\newtheorem{theo}[theorem]{Theorem}

\theoremstyle{definition}
\newtheorem{defi}[theorem]{Definition}
\newtheorem{rema}[theorem]{Remark}

\newtheorem{exam}[theorem]{Example}
\newtheorem{introthm}{Theorem}

\newcommand{\Z}{\mathbb{Z}}
\newcommand{\calO}{\mathcal{O}}
\newcommand{\p}{\mathfrak{p}}
\newcommand{\Q}{\mathbb{Q}}
\newcommand{\tr}{\mathrm{t}}
\newcommand{\K}{\mathcal{K}}
\newcommand{\R}{\mathcal{R}}
\newcommand{\N}{\mathbb{N}}
\newcommand{\Hom}{\mathrm{Hom}}

\newcommand{\rk}{\mathrm{rk}}
\newcommand{\F}{\mathbb{F}}
\newcommand{\mm}{\mathfrak{m}}

\newcommand{\EE}{\mathcal{E}}
\newcommand{\q}{\mathfrak{q}}

\newcommand{\coker}{\mathrm{coker}}
\newcommand{\CC}{\mathrm{cap}}

\def\mf{\mathfrak}
\def\mb{\mathbb}
\def\mc{\mathcal}

\newcommand{\hide}[1]{}

\title{Local Minkowski units in non-abelian extensions with cyclic Sylow $p$-subgroups}

\author{Wan Lee}
\address{   Department of Mathematics, Changwon National University, 20 Changwondaehakro, Uichang-gu, Changwon-si, Gyeongsangnam-do, 51140, South Korea}
\email{wannim@gmail.com}

\author{Donghyeok Lim}
\address{Korea National University of Education, 
Department of Mathematics Education, Cheongju, Korea}
\email{donghyeokklim@gmail.com}

\date{}

  \makeatletter
\def\@settitle{\begin{center}
 \normalfont \LARGE\bfseries \@title
\end{center}}

\def\section{\@startsection
  {section}{1}{\z@}
  {-2.5ex \@plus -1ex \@minus -.2ex}
  {1.3ex \@plus .2ex}
  {\normalfont\large\bfseries}}
\makeatother

\begin{document}

\begin{abstract}
We establish a criterion for the existence of a local Minkowski unit at $p$ that applies to all Galois extensions with Galois group isomorphic to the direct product of a non-$p$-group and a cyclic $p$-group. As applications, we construct non-abelian extensions admitting a local Minkowski unit at $p$ under various ramification conditions and analyze the Iwasawa module structure of units in $\mathbb{Z}_p$-extensions of number fields. We also extend our study to the case where the group of $p$-power roots of unity is nontrivial.
\end{abstract}

\maketitle

\tableofcontents

\section{Introduction}
 
Let $K/k$ be a finite Galois extension of number fields with Galois group $G=G_{K/k}$. Throughout this work, we assume that $K/k$ is unramified at all infinite places. Understanding the Galois module structure of unit group $U_K:=\mc O_K^\times$ of the ring $\calO_K$ of integers is one of the classical and enduring problems in algebraic number theory. By the classical Dirichlet–Herbrand theorem (cf. \cite[Thm. I.3.7]{GrasCFT}), its rational representation is completely determined by the isomorphism
\begin{equation*}
    \Q \otimes U_K \cong  \Q[G_{}] / (\tr_{G_{}})   \oplus \Q[G_{}]^{r_k},
\end{equation*}
where $r_k$ denotes the $\Z$-rank of $U_k$ and $\tr_{G_{}}=\sum_g g \in \Z[G_{}]$. In contrast, the $\Z[G_{}]$-structure of $U_K$ remains poorly understood, since its study involves both the arithmetic of $K/k$ and the integral representation theory of $G_{}$.

The existence of a \textit{strong Minkowski unit} when $r_k=0$ (i.e. $k=\Q$ or $k$ is imaginary quadratic) has been studied in several papers (cf. \cite{All, BouaLim, Brumer69, Burns2, Duval1, JT24, Marko05, Mars2, MazurUllom} and the references cited in \cite[Rem. 8.3]{JT24}). This question is equivalent to the following $\Z[G_{}]$-module isomorphism
\begin{equation*}
E_K := U_K/\mu_K \cong \Z[G_{}]/(\tr_{G_{}}),
\end{equation*}
 where $\mu_K$ denotes the group of roots of unity in $K$. 
The problem, despite its simple formulation, is far from trivial. We say that $K/k$ has a \textit{local Minkowski unit} at $p$ if the above isomorphism holds after applying $\Z_p \otimes_{} \bullet$, that is 
\begin{equation*}
\EE_K := \Z_p \otimes E_K \cong \Z_p[G_{}]/(\tr_{G_{}}) .
\end{equation*}
A major obstacle is the notoriously difficult classification of indecomposable $\Z_p[G_{}]$-lattices. Consequently, results on local Minkowski units for non-abelian Galois extensions have largely been confined to groups whose Sylow $p$-subgroup is cyclic of order $p$. 
By contrast, the existence of local Minkowski units for general abelian extensions is considerably better understood, thanks to the work of Fr\"{o}hlich  \cite{Frohlich2, Frohlich3} and Burns \cite{Burns1, Burns2}.

More recently, the Galois module structure of unit groups has attracted considerable attention for arbitrary base fields $k$. Since the $\Z_p$-rank of $\EE_K$ can be arbitrarily large compared to $[K:k]$, the Jordan--Zassenhaus theorem yields no finiteness restriction on its Galois module structure. Indeed, Kataoka and Ozaki \cite{OzakiKataoka} recently showed that every module arising in integral representation theory occurs, up to a free $\Z_p[G_{}]$-summand, as the Galois module of units of a finite Galois $p$-extension $K/k$.

Therefore, it is natural to study how the arithmetic of $K/k$ restricts the Krull--Schmidt decomposition of $\EE_K$. A basic invariant in this context is the $\Z_p$-rank $c(G_{},\EE_K)$ of its non-projective $\Z_p[G_{}]$-direct summands. The works \cite{Burns3, KumonLim} provide upper bounds for $c(G_{},\EE_K)$ for arbitrary finite Galois groups in terms of the $p$-rank of a certain $S$-class group of $K(\zeta_p)$ and the number of ramified primes in $K/k$. Consequently, if the $\Z_p$-rank of $\EE_K$ is sufficiently large relative to $c(G_{},\EE_K)$, then its Krull--Schmidt decomposition contains free $\Z_p[G_{}]$-direct summands. Such free direct summands are also of independent interest, since they play an important role in the study of tamely ramified pro-$p$-extensions of number fields. In particular, the $\F_p[G_{}]$-module structure of $\F_p\otimes_{}U_K$ plays an important role in the inverse Galois problem for the $p$-class field tower \cite{Ozaki3,HMR24} and, more recently, in the study of the deficiency of the Galois group of the $p$-class field towers \cite{HMR26}.

Further developments have been made in the case of $p$-extensions. In \cite{BLM23}, Yakovlev theory was applied to a certain family of pro-$p$ Galois extensions, including Galois extensions of local type (cf. \cite[Def. 10.9.6]{NSW}), yielding strong constraints on $c(G,\EE_K)$ as well as partial information on the non-projective component. More recently, using a different approach based on the Ritter--Weiss Tate sequences \cite{RW}, \cite{Burns4} obtained bounds for $c(G,\EE_K)$ that refine those of \cite{Burns3,KumonLim} while remaining valid for general Galois $p$-extensions. It is therefore natural to seek analogous results for more general finite groups.

\vskip 7pt

Against this background, we consider the Galois module structure of units for a new family of finite groups. More precisely, we show that when $G_{}=H \times P$, where $P$ is a cyclic $p$-group and $H$ (possibly non-abelian) has order prime to $p$, there is a simple criterion that sometimes completely determines the Krull-Schmidt decomposition of $\EE_K$.

\begin{introthm}\label{theo-main}
Let $K/k$ be a finite Galois extension with  Galois group $G_{K/k} =G  = P \times H$, where $P$ is a cyclic $p$-group and $H$ is a group whose order $|H|$ is coprime to $p$. Assume that no infinite place of $k$ is ramified in $K$. Then, we have the $\Z_p[G_{}]$-isomorphism
\begin{equation}\label{isoM}
\tag{M}
    \EE_K \cong \Z_p[G_{}]/(\tr_{G_{}}) \oplus \Z_p[G_{}]^{r_k}
\end{equation}
if and only if we have $H^1(P,E_K) \cong \Z/|P|\Z.$
\end{introthm}
\noindent Theorem~\ref{theo-main} shows that the isomorphism~\eqref{isoM} can be verified purely at the cohomological level. In particular, it extends to the above non-abelian groups a theorem of Burns~\cite[Thm. 3]{Burns2} that the existence of local Minkowski units for abelian number fields is governed by cohomology together with the class number ratios
appearing in Brauer--Kuroda type formulas.
(cf.\ Remarks~\ref{rema-factor} and~\ref{rema-Burns}).

\vskip 7pt

We next use Theorem~\ref{theo-main} to study non-abelian examples in which Isomorphism~\eqref{isoM} holds. Although representation-theoretic criteria are available, they usually require additional arithmetic information on ramification or class groups (cf.~\cite{BLM23}, \cite{Burns2}), and determining such arithmetic data explicitly is in general difficult.

The cohomological criterion in Theorem~\ref{theo-main} makes it possible to apply genus theory (cf. \eqref{sequence genus theory}), leading to the following result.

Throughout this paper, for any finite extension $M/L$ of number fields, let $\CC_{M/L} = \ker (Cl_L \to Cl_M)$ denote the kernel of the natural map (capitulation map), and let $h_M =|Cl_M|$ be the class number of $M$.

\begin{introthm}\label{theo-smallramification}
Let $K/k$ be a finite Galois extension with  Galois group $G_{K/k} =H \times P$, where no infinite places are ramified. Set $F=K^P$, the fixed field of $K$ by $P$. Assume
\begin{equation}\label{condC}
\tag{C} H^1(P, U_K) \cong H^1(P, E_K).
\end{equation}
\begin{enumerate}
    \item 
    If $K/F$ is unramified at all primes, Isomorphism \eqref{isoM} holds for $K/k$ if and only if we have $\CC_{K/F} \cong \Z/|P|\Z$.
    \item  
    Suppose that exactly one prime of $ F$ ramifies in $K/F$ and    $p\nmid h_k$. Then Isomorphism \eqref{isoM} holds for $K/k$ if and only if $\CC_{K/F}$ is trivial.
   \end{enumerate}
\end{introthm}

Theorem \ref{theo-smallramification} (i) also reveals a rather unexpected rigidity phenomenon for capitulation kernels of $p$-class groups in unramified cyclic $p$-extensions, especially in view of the generally complicated structure of $p$-class field towers. Under suitable hypotheses, the capitulation kernels at intermediate levels are determined by the capitulation kernel from the base field to the top field (Lemma \ref{assump-F12}). By Artin reciprocity, this translates into a result on transfer kernels in general FAb pro-$p$-groups (Corollary \ref{prop-tk}).

\vskip 7pt

We further extend these results to allow the group $\mu_{K,p}$ of $p$-power roots of unity of $K$ to be non-trivial. In that case, Condition~\eqref{condC} may fail, since~$\mu_{K,p}$ is not $P$-acyclic in general. For simplicity, many previous works assume $\zeta_p \notin K$ (cf.~\cite{BLM23, KumonLim, Belliard}), where~$\zeta_p$ denotes a primitive $p$-th root of unity. In \S\ref{sec-non-cohotri}, we investigate situations in which Condition~\eqref{condC} still holds even when $\mu_{K,p}$ is not $P$-acyclic (cf.~Theorem~\ref{theo-removingC}, Proposition~\ref{prop-C}). As a consequence, when $p=2$ and $K$ is totally real, Condition~\eqref{condC} follows from the assumptions of Theorem~\ref{theo-smallramification} even if \(\mu_{K,p}\) is not \(P\)-acyclic (Corollary~\ref{theoB-ext}).

\vskip 7pt

Since the setting of Theorem~\ref{theo-smallramification} arises naturally in $\Z_p$-extensions of number fields, we also apply it to Iwasawa modules arising from unit groups.

\begin{introthm}\label{theo-Iwasawa}
Let $F/k$ be a finite Galois extension of number fields in which no infinite places ramify.
Let $p$ be a prime with $p\nmid h_F \cdot [F:k]$.
Let $k_{\infty}/k$ be a $\Z_p$-extension such that $Fk_{\infty}/F$ is ramified at a unique prime. For each $n$, let $k_n$ denote the $n$-th layer of $k_{\infty}/k$, and set $F_n=Fk_n$.
Suppose that Condition \eqref{condC} holds for $F_n/F$ for each $n$. 
Then we have  isomorphisms
\begin{align*}
 & \varprojlim_{n \in \N} \, \EE_{F_n} \cong \Lambda[G_{F/k}]^{r_k+1} \  \text{ and } \\
 &   \varprojlim_{n \in \N} \, (\Z_p \otimes U_{F_n}) \cong \Big( \, \varprojlim_{n \in \N} \, \mu_{F_n, p} \, \Big )   \oplus\Lambda[G_{F/k}]^{r_k+1}
\end{align*}
of $\Lambda[G_{F/k}]$-modules, where the inverse limits are taken with respect to the norm maps and $\Lambda:=\Z_p[[G_{F_{\infty}/F}]]$ denotes the Iwasawa algebra of the Galois group $G_{F_{\infty}/F}$ of $F_{\infty}=Fk_{\infty}$ over $F$.
\end{introthm}

\noindent

Theorem \ref{theo-Iwasawa} complement previous works \cite{Belliard}, \cite{TMO13}, \cite[Ch.~11, \S 3]{NSW} on Iwasawa modules of units by covering certain non-abelian cases and a broader class of $\Z_p$-extensions (Remark~\ref{remark unit Iwasawa module history}). Furthermore, our work on the non-$P$-acyclic case also applies to various $\Z_p$-extensions of number fields containing $\zeta_p$ (cf. Corollary \ref{coro-IwaModul}).

\vskip 7pt

Finally, we apply our preceding results to construct new examples of local Minkowski units at~$p$. In general, one expects the Galois module structure of units to become more complicated as the number of ramified primes and the size of the $p$-class group increase. In particular, if~$K/k$ is a $p$-extension and~$r_k$ is sufficiently large compared to~$[K:k]$, the number of ramified primes, and the size of the $p$-class group, then~$\EE_K$ admits a free~$\Z_p[G_{K/k}]$-direct summand (cf.~\cite{Burns1, KumonLim, Burns4, BLM23}). 
On the other hand, when~$r_k=0$, the situation is more subtle, and local Minkowski units usually fail to exist for various reasons. This suggests that there is no uniform relationship between the size of the $p$-class group, ramification, and the existence of local Minkowski units. This is further illustrated by the following theorem proved in \S \ref{sec-manyram}, which shows that local Minkowski units at~$p$ may nevertheless exist even in the presence of large $p$-class groups and large ramification.

\begin{introthm}\label{theo-D}
   Let $F$ be a totally real cyclic number field and   $p$   an odd prime number.
   Assume  $p\nmid h_F\cdot [F:\Q]$
   and  $F \cap \Q(\zeta_p)=\Q$. 
   Then, for any given $r\ge 0$ and  $n\ge n_1 \ge n_2 \ge \cdots \ge n_r\ge 1  $, there exist infinitely many  cyclic extensions $L/\Q$ of degree $p^{n}$ such that:
   \begin{itemize}
   \item
      $Cl_K[p^\infty] \cong Cl_L[p^\infty] \cong \prod_{i=1}^{r}\Z/p^{n_i}\Z\,$, where  $K=LF$.
      \item
       $K$ admits a local Minkowski unit at $p$.
     \item
     $K/F$ is ramified at exactly $r+1$  primes.
           \end{itemize}
\end{introthm}

We conclude in \S\ref{sec-numexam} by presenting numerical examples of non-abelian number fields admitting local Minkowski units. All computations were carried out using PARI/GP~\cite{PARI2}. 

\bigskip

Throughout the paper,  we write $G_{M/L}$ for the Galois group of  a Galois extension $M/L$.
For a finite group $G$, we write $\tr_G:=\sum_{g\in G}g$ for its trace element.

 \subsection*{Acknowledgements}
Many of the results in this work were obtained through discussions with Zakariae Bouazzaoui. We are deeply grateful to him for his generous support and for permitting us to include these ideas in the present work. We greatly appreciate the valuable discussions with David Burns and Christian Maire. We are grateful to Andreas Nickel for pointing out a typo in an earlier version of this work. 
The first author is grateful to Yukako Kezuka for   earlier discussions about the elliptic case of the Leopoldt conjecture.
The second author would like to thank Zouhair Boughadi and Dominik Bullach for helpful comments, and Jilali Assim for his interest in this work. Part of this work was carried out during the second author's visit to the Institute for Advanced Studies in Mathematics (IASM) at Harbin Institute of Technology. 
The second author is grateful to IASM for its hospitality and support during his visit and to Oussama Hamza and Heon Lee for making his stay pleasant.

\subsection*{Funding}
The first author was supported by the National Research Foundation of Korea (NRF) grants funded by the Korea government (MSIT) (RS-2021-NR061985 and RS-2024-00451678) and by  Basic Science Research Program through the National Research Foundation of Korea (NRF) funded by the Ministry of Education (RS-2025-25435653).
The second author was supported by the National Research
Foundation of Korea (NRF) grants No. RS-2024-00462910.

\section{Preliminaries}

\subsection{Integral representations of a direct product}
In this subsection, we review basic notions in the theory of integral
representations of finite groups and recall a theorem of Jones that will play an important role. Throughout this subsection, we fix a finite group $G$ and a Dedekind domain $R$, with quotient field $K$.

\begin{defi}
A finitely generated $R[G]$-module that is torsion-free over $R$ is called an $R[G]$-lattice. An $R[G]$-lattice $M$ is called
\begin{itemize}
\item irreducible if $M$ does not contain a proper sublattice with smaller $R$-rank, 
\item indecomposable if $M$ is not a direct sum of two proper $R[G]$-sublattices.
\end{itemize}
\end{defi}

\begin{rema}\label{remark representation}
\begin{enumerate}
\item Every irreducible lattice is indecomposable, but the converse does not in general hold. When $R[G]$ is semisimple, the converse holds.

\item If $R$ is a discrete valuation ring whose residue field has characteristic coprime to $|G|$, then an $R[G]$-lattice $M$ is uniquely determined  by the $K[G]$-structure of $K \otimes_R M$ (cf. \cite[\S 15.5]{Serre}). 
\end{enumerate}
\end{rema}

\begin{exam}
The ring~$R[P]$ is a local ring when~$P$ is a finite $p$-group and~$R$ is a local ring whose residue field has characteristic~$p$. By Nakayama’s lemma, both~$R[P]$ and~$R[P]/(\tr_P)$ are indecomposable as~$R[P]$-lattices.
\end{exam}

The classification of indecomposable $R[G]$-lattices is known only for a very limited class of groups (cf. \cite[\S 15]{Reinersurvey}). Beyond these cases, the general classification problem is widely regarded as intractable (cf.~\cite[p.~241]{Duval1} for the case $G \cong (\Z/p\Z)^2$). In particular, a theorem of Heller and Reiner shows that unless the Sylow $p$-subgroup of $G$ is isomorphic to $\Z/p^i\Z$ with
$i\leq 2$, there exist infinitely many nonisomorphic indecomposable
$\Z_p[G]$-lattices.

In \cite{Jones}, Jones studied indecomposable $R_{\mathfrak p}[G]$-lattices, where $R_\mathfrak p$ denotes the localization of $R$ at a maximal ideal $\mf p$  and $G$ is a direct product $G_1\times G_2$. We use the following special case of his theorem.

\begin{theorem}(\cite[Thm.1]{Jones}) \label{Jones} 
Let $G=G_1 \times G_2$ be the direct product of two groups of coprime orders, and suppose that $|G_1|$ is prime to a prime $p$. Let $R$ be the ring of integers of a finite extension of $\Q_p$ that is a splitting field for $G_1$. Then every indecomposable $R[G]$-lattice is (isomorphic to) the outer tensor product of an irreducible $R[G_{1}]$-lattice with an indecomposable $R[G_{2}]$-lattice.
\end{theorem}

\begin{rema}\label{remark indecomposable tensor}
Let $R$ and $G$ be as in Theorem \ref{Jones}. Let $\mathfrak{N}$ be an irreducible $R[G_1]$-lattice and $\mathfrak{M}$ an indecomposable $R[G_2]$-lattice. Then $\mathfrak{N} \otimes_R \mathfrak{M}$ is an indecomposable $R[G]$-lattice. Suppose not, by Theorem \ref{Jones}, the Krull-Schmidt decomposition of $\mathfrak{N} \otimes_R \mathfrak{M}$ contains a summand of the form $\mathfrak{N} \otimes_R \mathfrak{M}'$ for some indecomposable $R[G_2]$-lattice $\mathfrak{M}'$. 
By the Krull-Schmidt theorem, we must then have $\mathfrak{M} \cong \mathfrak{M}'$.
\end{rema}

\subsection{Galois cohomology of unit group}\label{subsec-GC}

We recall preliminary results in Galois cohomology. Throughout \S \ref{subsec-GC} and \ref{subsec-descent}, we fix a Galois extension $M/L$ of number fields with the Galois group $G$. We write $P_M$ and $I_M$ for the group of principal fractional ideals of $M$ and fractional ideals of $M$, respectively.

\begin{lemm}\label{lemma Iwasawa theorem}(cf. \cite{Iwasawa})
There exists an isomorphism $H^1(G,U_M) \cong P_M^{G}/P_L.$ Hence, $H^1(G,U_M)$ identifies with the kernel of the natural map $I_M^{G}/P_L \longrightarrow Cl_M.$
\end{lemm}

We recall the well-known unit principal genus theorem.

\begin{prop}\label{unit principal genus theorem}
Suppose that  $M/L$ is a cyclic extension  unramified at all infinite places. Then one has $|H^2(G, U_M)| \cdot [M:L] = |H^1(G, U_M)|$.
\end{prop}

\begin{lemm}\label{lemma tensor cohomology}(cf. \cite[Lem.~2.1]{LeeSeo})
Let $A$ be a $G$-module.
Then we have 
\begin{enumerate}
    \item 
$H^i(G, \Z_p \otimes A) \cong \Z_p \otimes H^i(G, A)$ for  all $i\ge 0$;  
    \item 
    $\widehat{H}^i(G, \Z_p \otimes A) \cong \Z_p \otimes \widehat{H}^i(G, A)$ for all  $i$, where $\widehat{H}^{i}(G, A)$ denotes Tate cohomology.  
\end{enumerate}
\end{lemm}
\noindent
If the Sylow $p$-subgroup  $\mu_{M, p}$  of $\mu_M$ is $G$-acyclic, then by Lemma~\ref{lemma tensor cohomology}, there is an isomorphism
\begin{equation*} \label{iso unit tensor}
    H^1(G, \Z_p \otimes U_M) \cong H^1(G, \EE_M).
\end{equation*}
A sufficient condition for this is that~$\mu_{M,p}$ be $G$-cohomologically trivial, for which we recall the following criterion (cf. \cite[Lem. 5.4.4(1)]{Popescu}, \cite[Lem.~9.1.4]{NSW}).

\begin{lemm}\label{coho-trivial}
The following statements hold.
\begin{enumerate}
\item Suppose that $p\neq 2$.
Then $\mu_{M,p}$ is $G$-cohomologically trivial if and only if $\mu_{M,p}=1$ or $p \nmid [M:L(\mu_{M,p})]$.
\item Suppose that $p = 2$. 
 \begin{itemize}
\item[(a)] 
Then $\mu_{M,2}$ is $G$-cohomologically trivial if and only if $2 \nmid [M:L(\mu_{M,2})]$ and, in addition, when $\mu_{M,2} \neq \mu_{L,2}$, the field $L \cap \Q(\mu_{M,2})$ is not totally real.
\item[(b)] 
Moreover, if $\mu_{M,2} \neq \mu_{L,2}$, $M=L(\mu_{M,2})$ and
$L \cap \Q(\mu_{M,2})$ is totally real, we have
$$ G\cong G_{M/L(\sqrt{-1})} \times G_{L(\sqrt{-1})/L} \ \ \text{ and } \  \  H^1(G, \mu_{M,2}) \cong \Z/2\Z.$$
 \end{itemize}
\end{enumerate}
\end{lemm}

\noindent We include a proof of the following well-known lemma for completeness.
\begin{lemm}\label{lemma Brauer group}
If $M/L$ is unramified at all but possibly one place, then we have $H^1(G, P_M)\cong H^2(G, U_M)$.
\end{lemm}

\begin{proof}
Applying $G$-cohomology to the  short exact sequence 
$0\rightarrow U_M \rightarrow M^\times \rightarrow P_M\rightarrow 0$ yields 
the exact sequence $0\rightarrow H^1(G, P_M) \rightarrow H^2(G, U_M) \rightarrow \mathrm{Br}(M/L)$,  where $ \mathrm{Br}(M/L) := H^2(G, M^\times )$.
It suffices to prove that the map  $H^2(G, U_M) \rightarrow \mathrm{Br}( M/L)$ is zero.

We have the following  commutative diagram
$$
\begin{tikzcd}
&   H^2\left(G,  U_M  \right) \ar{d}\ar{r}{ \oplus_v{res}_vU} &\bigoplus_v H^2\left(G_v, U_{M_v}\right) \ar{d} & &\\
&\mathrm{Br}(M/L) \ar{r}{\oplus_v{res}_v M} \ar[hook]{d}{{inf}} &\bigoplus_v\mathrm{Br}(M_v/L_v) \ar[hook]{d}{\oplus_v{inf}_v}& & \\
 0 \ar{r}&  \mathrm{Br}(L) \ar{r} &\bigoplus _v\mathrm{Br}(L_v)\ar{r}{\oplus_v {inv}_v} &\mathbb{Q} /\mathbb{Z} \ar{r} &0,
\end{tikzcd}
$$
 where $v$ runs through all places of $L$, $G_v$ denotes the decomposition subgroup of $G$ at $v$, ${res}_vU$ and ${res}_vM$ are appropriate restriction maps, ${inf}$ and ${inf}_v$ are appropriate inflation maps,  and $U_{M_v}$ denotes the group of local units of $M$ at $v$. 
 (Here we use the same symbol $v$ for a prime of $L$ and a prime of $M$ above it, as no confusion arises.) 
 The bottom row is the fundamental exact sequence of class field theory on Brauer groups (see \cite[Thm.~8.17]{NSW}).
 
Note that the vertical maps ${inf}$ and ${inf}_v$'s are injective for each $v$ by the inflation-restriction exact sequence with the help of Hilbert's theorem $90$.
Furthermore, since the group of local units $U_{M_v}$ is $G_v$-cohomologically trivial for  unramified $v$'s, we have $\bigoplus_v H^2\left(G_v, U_{M_v}\right)=H^2(G_{v'}, U_{M_{v'}})$ for some $v'$ by the assumption.
Now  the claim follows from the above diagram by noting that $\mathrm{Br}(L_v)\stackrel{{inv}_v}\longrightarrow \mathbb{Q}/\mathbb{Z}$ is injective for each $v$.
\end{proof}

\subsection{Galois descent of unit lattice}\label{subsec-descent}

We discuss the Galois descent of unit lattices. The short exact sequence 
 $$
 0\longrightarrow \mu_M \longrightarrow U_M \longrightarrow E_M \longrightarrow 0
 $$
  induces an injection $E_L \hookrightarrow E^{G}_M$, and hence, by Lemma \ref{lemma tensor cohomology}, an injection $\EE_L \hookrightarrow \EE^{G}_M$. More precisely, one has the isomorphism  (cf. \cite[Lem. 2.14]{Bartel12})
\begin{equation}\label{remark unit lattice galois fixed} 
 \EE^{G}_M / \EE_L  \cong    \ker \big (H^1(G, \mu_{M,p}) \to H^1(G, \Z_p \otimes U_M) \big )  .
\end{equation}    
 It follows that Galois descent $\EE_M^G=\EE_L$ holds when $\mu_{M,p}$ is $G$-acyclic. We also recall the following well-known lemma.

\begin{lemm}\label{remark subfield Minkowski}
Let $N$ be an intermediate subfield of $M/L$ with $H=G_{M/N}$.
\begin{enumerate}
    \item 
If Isomorphism \eqref{isoM}  for $M/L$ holds, then we have $\EE_N=(\EE_M)^H=\tr_H \EE_M$.
    \item 
     If $N/L$ is moreover Galois, then Isomorphism \eqref{isoM}   for $N/L$ holds.
\end{enumerate}
\end{lemm}

\begin{proof}
Applying $H$-cohomology to the short exact sequence
   \begin{equation*}
       0\longrightarrow (\tr_{G}) \longrightarrow \mb Z_p[G] \longrightarrow  \Z_p[G]/(\tr_G) \rightarrow 0,
      \end{equation*}
we obtain
\begin{equation}\label{equation 0 tate cohomology}
   (\Z_p[G])^H=\tr_H\cdot \Z_p[G] \  \text{ and } \   \big ( \Z_p[G]/(\tr_G )\big )^{H} =  \big ( \tr_H \cdot\Z_p[G] \big )/(\tr_G )  
\end{equation}
   by noting  $H^1(H, (\tr_G ))=0$ and $(\tr_G)^H =(\tr_G)$.
If $H$ is moreover normal in $G$ with $\bar{G}:=G/H$, then we have 
\begin{equation} \label{equation trace group ring}
\tr_H \cdot \Z_p[G] \cong \Z_p[\bar{G}]   \ \text{ and } \  \big(\tr_H \cdot \Z_p[G]  \big) /(\tr_G)   \cong \Z_p[\bar{G}]/(\tr_{\bar{G}}) 
\end{equation}
as $\Z_p[\bar{G}]$-lattices.
 
Assume that Isomorphism \eqref{isoM} for $M/L$ holds.
Let $\theta_M: \EE_M \to \Z_p[G]/(\tr_G) \oplus \Z_p[G]^{r_L}$ be a $\Z_p[G]$-isomorphism.  
Then, we have the commutative diagram
\begin{equation}\label{dia-compatib}
    \begin{tikzcd}
\EE_M \arrow[d, "\theta_M"] \arrow[r, "N_{M/N}"] & \EE_N \arrow[r, hook] & \EE_M \arrow[d, "\theta_M"] \\
{\Z_p[G]/(\tr_G) \oplus \Z_p[G]^{r_L}} \arrow[rr, "\tr_H"]            &                       & {\Z_p[G]/(\tr_G) \oplus \Z_p[G]^{r_L}}           
\end{tikzcd}
\end{equation}

From Diagram~\eqref{dia-compatib} and Equation~\eqref{equation 0 tate cohomology}, we deduce that $\EE_N=\EE^{H}_M = \tr_H \EE_M$. Hence, the norm map $N_{M/N} : \EE_M \to \EE_N$ is surjective. Moreover, if $N/L$ is Galois, then $N/L$ also satisfies Isomorphism \eqref{isoM} by Equation~\eqref{equation trace group ring}.  
\end{proof}

Therefore, the Galois descent of the unit lattice is a necessary condition for   Isomorphism \eqref{isoM}. 
 We return to the problem of  Galois descent in the case where  $\mu_{M,p}$ is not acyclic in \S \ref{sec-non-cohotri}.

\section{Cohomological determination of Galois modules}

In this section, we prove Theorem \ref{theo-main} and give several remarks.

\subsection{A representation theoretic proposition}

Let $G = H \times P$ be a fixed finite group, where $P$ is a $p$-group and $H$ is a finite group with $p \nmid |H|$. Let $M$ be a $\Z_p[G]$-lattice such that 
$$ \Q_p \otimes_{\Z_p} M \cong \Q_p[G]/(\tr_G) \oplus \Q_p[G]^n$$
as $\Q_p[G]$-modules for some non-negative integer $n$. Let $\K/\Q_p$ be a finite extension that is a splitting field for $H$, and let $\R$ denote its ring of integers. Without loss of generality, we may assume that $\K/\Q_p$ is unramified (cf. \cite[\S 12.3]{Serre}).

Let $\mm$ and $\mm_{\R}$ denote the maximal ideals of $\Z_p[P]$ and $\R[P]$, respectively. 
 For each irreducible character $\chi$ of $H$ over $\K$, let $N_\chi$ be the irreducible $\R[H]$-lattice such that $\K \otimes_{\R} N_{\chi}$ affords $\chi$ (cf. Remark \ref{remark representation} (ii)). 
Let $m_\chi$ denote the multiplicity of $\chi$ in the regular representation $\K[H]$.
Note that if $P=1$, then the assumption on $\Q_p \otimes_{\Z_p} M$ together with Remark~\ref{remark representation} (ii) implies $M \cong \Z_p[G]/(\tr_G) \oplus \Z_p[G]^{n}$. 
Thus, we will assume that $P$ is non-trivial.

For any finitely generated $\R$-module $V$, let $\rk_\R V$ denote the $\K$-dimension of $\K \otimes_{\R} V$. If $V$ is a finitely generated $\R[P]$-module, let  $d_{\R[P]} V$ denote the minimal number of generators of $V$ over $\R[P]$. By Nakayama's lemma, this is equal to the dimension of $ V/\mm_{\R} V$ over the residue field $\R/p\R $.

The main result of this subsection is the following representation-theoretic proposition.

\begin{prop}\label{prop-repre}
The lattice $M$ is isomorphic to $\Z_p[G]/(\tr_G) \oplus \Z_p[G]^n$ if and only if we have $M/\mm M \cong \F^{(n+1)|H|}_p$.
\end{prop}

To apply Theorem \ref{Jones}, we set $M_{\R}:=\R \otimes_{\Z_p} M$. Since $\R$ is free over $\Z_p$, for any $\Z_p[G]$-lattice $N$, one has $\R \otimes_{\Z_p} N \cong N^{[\K:\Q_p]}$ as $\Z_p[G]$-lattices. By the Krull-Schmidt theorem, the condition $M \cong \Z_p[G]/(\tr_G) \oplus \Z_p[G]^n$ is therefore equivalent to $M_{\R} \cong \R[G]/(\tr_G ) \oplus \R[G]^n$. 

Since $\K/\Q_p$ is unramified, we have $\mm_{\R}=\mm \R[P]$.
Tensoring the exact sequence
\begin{equation*}
    0 \longrightarrow \mm M \longrightarrow M \longrightarrow M/\mm M \longrightarrow 0
\end{equation*}
with $\R \otimes_{\Z_p} \bullet$, we obtain an isomorphism $ \R \otimes_{\Z_p} \big ( M/\mm M \big ) \cong M_{\R}/\mm_{\R} M_{\R}$ and the identity $d_{\R[P]} M_{\R} = \mathrm{dim}_{\F_p} M/\mm M.$ Consequently, Proposition \ref{prop-repre} reduces to the following lemma.

\begin{lemm}\label{reductionlemma}
We have $M_{\R} \cong \R[G]/(\tr_G) \oplus \R[G]^n$ if and only if $d_{\R[P]} (M_{\R})=(n+1)|H|$.    
\end{lemm}

For each irreducible character $\chi$ of $H$ over $\K$, let $e_{\chi} \in \R[H]$ be the central primitive idempotent associated with $\chi$.

\begin{lemm}\label{lemm-KSD}
As an $\R[G]$-lattice, the module $\R[G]/(\tr_G)$ admits the Krull-Schmidt decomposition
\begin{equation*}
     \R[P]/(\tr_P) \oplus \bigoplus_{\chi \neq 1} \big ( N_{\chi} \otimes \R[P] \big )^{m_{\chi}},
\end{equation*}
where $\chi$ runs over the non-trivial irreducible characters of $H$ over $\K$.
\end{lemm}

\begin{proof}
This follows since, for $\chi \neq 1$, we have $$e_{\chi} \big (\R[G]/(\tr_G) \big ) \cong e_{\chi} \R[G]/ e_{\chi} \tr_G\R[G]=e_{\chi}\R[G]$$ and $$e_{\chi}\R[G] \cong e_{\chi} \big (\R[H] \otimes_{\R} \R[P] \big ) \cong N^{m_{\chi}}_{\chi} \otimes_{\R} \R[P] \cong \big ( N_{\chi} \otimes_\R \R[P] \big)^{m_{\chi}}.$$ 
The case $\chi=1$, yielding $\R[P]/(\tr_P) \cong N_1 \otimes_{\R} \big(\R[P]/(\tr_P)\big)$, is similar (cf. Equation~\eqref{equation trace group ring}). The indecomposability of the summands in the statement follows from Remark~\ref{remark indecomposable tensor}.
\end{proof}

According to Theorem \ref{Jones}, $M_{\R}$ admits a Krull-Schmidt decomposition of the form
\begin{equation}\label{KS-decomp}
    M_{\R} \cong \bigoplus_{\psi, \chi} \big ( N_{\chi} \otimes_{\R} M_{\psi} \big )^{r_{\chi, \psi}}
\end{equation}
where $\chi$ and $\psi$ run over the irreducible $\R[H]$-lattices and indecomposable $\R[P]$-lattices, respectively, and $r_{\chi, \psi} \in \Z_{\geq 0}$.
From
\begin{equation} \label{equation K[P]-module structure}
    \K \otimes_\R \big ( e_{\chi} M_{\R}) \cong e_{\chi} \big ( \K \otimes_{\R} M_{\R} \big ) \cong e_{\chi} \big ( \K[G]/(\tr_G) \oplus \K[G]^n \big ),
\end{equation}
  we obtain, for each $\chi$, 
\begin{equation}\label{rk-rational}
   (n+1) \cdot m_{\chi} \cdot \rk_\R N_{\chi} \cdot |P| - \delta_{\chi,1} = \rk_\R \big ( e_{\chi} M_\R \big ) = \rk_{\R} N_{\chi} \cdot \bigg ( \sum_{\psi} r_{\chi, \psi} \cdot \rk_{\R} M_{\psi} \bigg ), 
\end{equation}
where $\delta_{\chi,1}=0$ if $\chi$ is non-trivial and $1$ otherwise.

\begin{proof}[\textbf{Proof of Lemma~\ref{reductionlemma}}]
The necessity follows easily from Lemma \ref{lemm-KSD} together with the identity
\begin{equation}\label{non-p identity}
    |H| =\sum_{\chi} m_{\chi} \cdot \rk_{\R} N_{\chi}.
\end{equation}
\noindent The sufficiency follows from the following two claims, which we prove in turn.
\vskip 5pt
\noindent \textbf{Claim 1}: If $d_{\R[P]} M_{\R} = (n+1) \cdot|H|$, then for each $\chi$ we have $$d_{\R[P]} (e_{\chi}M_{\R}) = (n+1) \cdot m_{\chi} \cdot \rk_\R N_{\chi}.$$
\vskip 5pt
By the identity $\sum_{\chi} d_{\R[P]} (e_{\chi} M_{\R}) = d_{\R[P]} (M_{\R})$ and \eqref{non-p identity}, if the claim were false, then there would exist some $\chi_0$ such that $d_{\R[P]}(e_{\chi_0}M_{\R}) < (n+1) \cdot m_{\chi_0} \cdot \rk_\R N_{\chi_0}.$ By Nakayama's lemma, this yields a surjective $\R[P]$-module homomorphism
\begin{equation*}
    \R[P]^{(n+1) \cdot m_{\chi_0} \cdot \rk_\R N_{\chi_0} -1} \longrightarrow e_{\chi_0}M_{\R}
\end{equation*}
and hence $\rk_\R(e_{\chi_0} M_\R) \leq |P| \cdot \big ( (n+1) \cdot m_{\chi_0} \cdot \rk_\R N_{\chi_0}-1 \big )$, which contradicts \eqref{rk-rational}.

\vskip 5pt
\noindent \textbf{Claim 2}: For a given $\chi$, if $d_{\R[P]} (e_{\chi} M_{\R}) = (n+1) \cdot m_{\chi} \cdot \rk_{\R} N_{\chi}$, then we have

\begin{equation*}
e_{\chi} M_\R \cong
\begin{cases}
\bigl( N_{\chi} \otimes_\R \R[P] \bigr)^{(n+1)m_{\chi}}
& (\chi \neq 1), \\[12pt]
\R[P]/(\tr_P) \oplus \R[P]^n
& (\chi = 1).
\end{cases}
\end{equation*}
\vskip 5pt
By the assumption, we have $d_{\R[P]} \big ( e_{1} M_\R \big ) = n+1$. 
By Nakayama's lemma, we obtain a short exact sequence $0 \to N \to \R[P]^{n+1} \to e_1 M_\R \to 0$ of $\R[P]$-modules, where $N \cong \R$ as $\R[P]$-modules by \eqref{equation K[P]-module structure}.
Considering the image of an $\R$-generator of $N$ in $\R[P]^{n+1}$, we obtain an exact sequence
\begin{equation*}
    0 \longrightarrow \R[P]^n \longrightarrow e_1M_\R \longrightarrow W \longrightarrow 0,
\end{equation*}
where $W$ is $\R$-torsion-free.
Taking $\Hom_{\R}(-,\R)$ and using that $\R[P] \cong \Hom_\R(\R[P], \R)$ is projective, we obtain $e_1M_\R \cong \R[P]^{n} \oplus W$ as $\R[P]$-modules.
 The $\K[P]$-module structure of $\K \otimes_\R (e_1M_\R)$ implies that $\tr_P$ annihilates $W$. A comparison of $\R$-ranks then yields $W \cong \R[P]/(\tr_P).$

If $\chi \neq 1$, then by \eqref{KS-decomp} we have
\begin{equation}\label{equality-coinvariance}
d_{\R[P]}(e_{\chi}M_{\R}) = \rk_{\R} N_{\chi} \cdot \bigg ( \sum_{\psi} r_{\chi, \psi} \cdot d_{\R[P]}(M_{\psi}) \bigg ). 
\end{equation}
The assumption together with \eqref{rk-rational} implies  
$$|P| \cdot d_{\R[P]} (e_{\chi} M_\R) = \rk_\R (e_{\chi} M_\R) =\rk_{\R} N_{\chi} \cdot \bigg ( \sum_{\psi} r_{\chi, \psi} \cdot \rk_{\R} M_{\psi} \bigg ).
$$
Substituting \eqref{equality-coinvariance} into this identity, we obtain
\begin{equation*}
    \sum_{\psi} r_{\chi, \psi} \cdot  d_{\R[P]}(M_{\psi}) \cdot |P| = \sum_{\psi} r_{\chi, \psi} \cdot \rk_{\R}(M_{\psi}). 
\end{equation*}
Since $d_{\R[P]}(M_{\psi}) \cdot |P| \geq \rk_{\R} (M_{\psi})$ for each $\psi$, equality holds for all $\psi$  with $r_{\chi, \psi}\neq 0$.
Hence the surjection $\R[P]^{d_{\R[P]}(M_{\psi})} \to M_{\psi}$, which exists by Nakayama's Lemma, is an isomorphism.
By indecomposability of $M_{\psi}$, it follows that $M_{\psi} \cong \R[P]$. Finally, \eqref{rk-rational} shows that the multiplicity of $N_{\chi} \otimes_{\R} \R[P]$ in $M_{\R}$ is  $(n+1)m_{\chi}$.
\end{proof}

\subsection{Proof of Theorem \ref{theo-main}}\label{sec-A}

Let $G$ be as in the previous subsection.
In addition, assume that $P$ is cyclic.
By Proposition \ref{prop-repre}, we obtain the following theorem.

\begin{theo}[Theorem \ref{theo-main}]\label{theoAgeneral}
Let $K/k$ be a finite Galois extension with $G_{K/k}=G = P \times H$. 
Assume that no infinite place of $k$ is ramified in $K$. 
Then Isomorphism~\eqref{isoM} holds for $K/k$, namely,
\begin{equation*}\tag{M}
\EE_K \cong \Z_p[G]/(\tr_G) \oplus \Z_p[G]^{r_k}
\end{equation*}
if and only if $H^1(P, \EE_K) \cong \Z/|P|\Z$.    
\end{theo}

\begin{proof}
Recall $ \EE_K= \Z_p \otimes E_K$.
The necessity follows from
\begin{equation*}
 H^1(P, E_K) \cong H^1(P, \EE_K)  \cong H^1(P, \Z_p[P]/(\tr_P)) \cong \Z/|P|\Z.
\end{equation*}
Here, the second isomorphism follows from  $\EE_K \cong \Z_p[P]^{(r_k+1)|H|-1} \oplus \Z_p[P]/(\tr_P)$ as $\Z_p[P]$-modules and the $P$-cohomological triviality of $\mb Z_p[P]$.  The third follows from  $$H^1(P, \Z_p[P]/(\tr_P))\cong H^2(P, (\tr_P))\cong \mb Z/|P|\mb Z.$$

It remains to prove the sufficiency.
The long exact sequence in cohomology associated with the sequence 
\begin{equation*}
1\longrightarrow E_{K}\xrightarrow{ \,\,p \,\,} E_{K}\longrightarrow E_{K}/E_{K}^{p}\longrightarrow1
\end{equation*}
yields the exact sequence 
\begin{equation*}
0\longrightarrow E_{K}^P/\left(E_{K}^P\right)^p\longrightarrow (E_{K}/E_{K}^{p})^{P}\longrightarrow H^{1}(P,E_K)[p]\longrightarrow0,
\end{equation*}
where $H^1(P,E_K)[p]$ denotes the subgroup of $H^1(P,E_K)$ consisting of elements annihilated by $p$.

Let $F=K^P$ be the fixed subfield of $P$.
Since $F/k$ is unramified at the infinite places, we have $|E_F/E^p_F|=p^{(r_k+1) \cdot|H|-1}$. 
Since $[E_K^P:E_F]$ is finite (cf. \S \ref{subsec-descent}), it follows that $|E_K^P/(E_K^P)^p|=p^{(r_k+1) \cdot|H|-1}$. 
Moreover, by assumption, we have $H^1(P, E_K)[p] \cong \F_p$, and hence $(E_{K}/E_{K}^{p})^{P} \cong \F^{(r_k+1) \cdot|H|}_p$.

Since $E_{K}/E_{K}^{p}$ is finite and $P$ is cyclic, the groups $(E_{K}/E_{K}^{p})^{P}$ and the coinvariance $(E_{K}/E_{K}^{p})_{P}$ of $E_K/E^p_K$ by $P$ have the same order.
Since $(E_{K}/E_{K}^{p})_{P} \cong \EE_K/\mm \EE_K$, the sufficiency follows from Proposition \ref{prop-repre} and the Dirichlet-Herbrand theorem.
\end{proof}

\begin{coro}
Let $K/k$ be as above and assume that $k$ is either rational or imaginary quadratic. The field $K$ has a local Minkowski unit at $p$ if and only if $H^1(P, E_K) \cong \Z/|P|\Z$.
\end{coro}

\begin{rema}\label{rema-factor} 
In general, for non-cyclic Galois groups, Brauer–Kuroda type formulas for the quotient of class numbers of subfields (\cite[Thm. A]{BouaLim}) arise as necessary conditions for the existence of a local Minkowski unit, as they determine the factor equivalence class of the unit lattice (cf. \cite{Bartel12, Bartel14, Burns2}).
However, in our setting, for group-theoretic reasons, any two $\Z_p[G]$-lattices with the same self-dual rational representation are factor equivalent, even when $H$ is non-cyclic, by \cite[Prop. 3.9]{Bartel12} and \cite[Cor. 2.12]{Bartel14}. 
In contrast, if $P$ is not cyclic, then the Brauer-Kuroda type formula arises as a necessary condition for the existence of local Minkowski units (cf. \cite[\S 7]{BouaLim}). 
\end{rema}

\begin{rema}\label{rema-Burns}
The factor equivalence class of a lattice~$M$ encodes coarser information than its full $\Z_p[G]$-module structure. However, for abelian extensions of $\Q$ or imaginary quadratic fields, Burns \cite[Thm. 3]{Burns2} showed that the factor equivalence class of the unit lattice, together with its Galois cohomology, suffices for the existence of a local Minkowski unit (cf. \cite[\S 3.2]{BouaLim}). In view of the preceding remark, our result can be viewed as a non-abelian generalization of this phenomenon.
\end{rema}

\begin{rema}
Since $G$ has a cyclic Sylow $p$-subgroup, Proposition~\ref{prop-repre} may be compared with Yakovlev's theorem~\cite{Yakovlev}. The latter requires knowledge of the groups $H^1(J,M)$ for all subgroups $J$ of a Sylow $p$-subgroup, together with the restriction and corestriction maps between them, whereas Proposition~\ref{prop-repre} shows that in our more special setting, it suffices to know only the cohomology of the Sylow $p$-subgroup itself.

An observation of Torzewski for dihedral groups $G = D_{2p}$ \cite[Lem.~3.4.1]{Torzewskithesis} shows that, under the condition $\Q_p \otimes_{\Z_p} M \cong \Q_p[G]/(\tr_G)$, the $\Z_p[D_{2p}]$-structure of $M$ is determined by the full $\F_p[D_{2p}/C_p]$-module structure of $H^1(C_p, M)$. While this shows that our result does not extend to arbitrary non-abelian groups $G$ with a cyclic Sylow $p$-subgroup, it also suggests that for semidirect products $G = P \rtimes H$, where $P$ is cyclic of $p$-power order and $(|H|,p)=1$, the existence of a local Minkowski unit may be governed by the $G/P$-module structure of $H^1(P,E_K)$.
\end{rema}

\begin{rema}
In \cite{HMR24}, a Galois extension $M/L$ is said to have $n$ Minkowski units (at $p$) if the $\F_p[G_{M/L}]$-module $U_M/U_M^p$ admits a free
$\F_p[G_{M/L}]$-direct summand of rank $n$. This notion has important applications to the study of tamely ramified pro-$p$-extensions. Theorem~\ref{theo-main} complements earlier results (\cite{BLM23, Burns4}, \cite[\S 5]{HMR21}) on the existence of Minkowski units (at $p$).
\end{rema}

\section{Proof of Theorem \ref{theo-smallramification}}\label{sec-B}
We retain the notation and assumptions of Theorem~\ref{theo-main}. Recall that $h_k$ denotes the class number of $k$, and that Condition~\eqref{condC} asserts that $H^1(P,U_K) \cong H^1(P,E_K)$.

\begin{theo}[Theorem \ref{theo-smallramification}]\label{theorem final}
Let $K/k$ be as in Theorem \ref{theo-main}, and let $F:=K^P$ denote the fixed field of $K$ by $P$. Assume that Condition \eqref{condC} holds.
\begin{enumerate}
    \item If $K/F$ is unramified, then Isomorphism \eqref{isoM} holds for $K/k$ if and only if $\CC_{K/F} \cong \Z/|P|\Z$. In this case, the capitulation map $Cl_F \to Cl^P_K$ is surjective.

    \item If $K/F$ is ramified at a unique prime of $F$ and $p\nmid h_k$, then the following are equivalent:

     \begin{itemize}
         \item Isomorphism \eqref{isoM} holds for $K/k$;
         \item 
          $\CC_{K/F}=0$;
         \item 
         $\coker(Cl_F\rightarrow Cl_K^P)=0$.
     \end{itemize}
   \end{enumerate}
\end{theo}
\noindent

\begin{proof}
Applying $P$-cohomology to the short exact sequence $0\rightarrow P_K \rightarrow I_K\rightarrow Cl_K \rightarrow 0$, we obtain the following exact commutative diagram
$$
\begin{tikzcd}[column sep=2em, row sep=1.5em]
0 \ar{r}&P_F \ar{r} \ar{d} & I_F\ar{d}  \ar{r}& Cl_F  \ar{d} \ar{r}& 0 &\\
 0 \ar{r}&  P_K^P \ar{r} & I_K^P\ar{r} & Cl_K^P\ar{r} &H^1(P, P_K) \ar{r} &0.
\end{tikzcd}
$$
Recall that $H^1(P, U_K)\cong P_K^P/P_F$ and $H^2(P, U_K)\cong H^1(P, P_K)$ by Lemma~\ref{lemma Iwasawa theorem} and Lemma~\ref{lemma Brauer group}, respectively.
Furthermore, we have $I_K^P/I_F \cong \bigoplus_v \mathbb{Z}/e_v\mathbb{Z}$, where $e_v$ denotes the ramification index of $v$ in $K/F$.
Applying the snake lemma to the above diagram yields the exact sequence
	\begin{equation}\label{sequence genus theory}
	\xymatrix@=1pc @C=0.5cm @R=7.5mm{
		&0 \ar[r]& \CC_{K/F}  \ar[r]& H^1(P,U_K) \ar[r] & \bigoplus_{v} \mathbb{Z}/e_v\mathbb{Z} \hspace*{-0.75cm} &  \ar `r[rd]
		`d[ll]
		`l[lllld]
		`d[lll]
		[llld] & \\
		& & \mathrm{coker}(Cl_F \to Cl^P_K) \ar[r] & H^1(P,P_K) \ar[r]& 0. \hspace*{2cm} &  & }
\end{equation}

If $K/F$ is unramified, then, by \eqref{sequence genus theory}, we have 
\begin{equation}\label{equation unramified kernel cokernel}
\CC_{K/F} \cong H^1(P, U_K) \ \ \hbox{ and }\ \  \coker(Cl_F \rightarrow Cl_K^P) \cong H^2(P, U_K).
\end{equation}
Hence, applying Theorem~\ref{theo-main} together with Proposition~\ref{unit principal genus theorem}, we obtain claim (i).
\vskip 5pt
Suppose that exactly one prime $\mathfrak{l}_F$ of $F$ ramifies in $K/F$.
Since $|H|$ and $|P|$ are relatively prime, the cyclic extension $K^{H}/k$ is unramified outside $\mf l:=\mathfrak{l}_F \cap \calO_k$, and $\mathfrak{l}_F$ is the unique prime of $F$ lying above $\mathfrak{l}$. 
As $p\nmid h_k$, $\mathfrak{l}$ is totally ramified in $K^{H}/k$, and hence $\mathfrak{l}_F$ is totally ramified in $K/F$. 
Consequently, there is a unique prime ideal $\mathfrak{L}$ of $\mc O_K$ above $\mathfrak{l}$.
It follows that the class of $\mathfrak{L}$ generates $I_K^P/I_F\cong   \bigoplus_v \mathbb{Z}/e_v\mathbb{Z}=\mathbb{Z}/|P|\mathbb{Z}$.

Since $K^H/k$ is a $p$-extension unramified outside $\mathfrak{l}$, we have $p\nmid h_{K^H}$. Moreover, as $|H|$ is prime to $p$, $\mathfrak{L}^{t} $ is principal for some integer $t$ coprime to $p$.

The class of $\mathfrak{L}^{t}$ generates $I^P_K/I_F$, so the map $H^1(P, U_K) \rightarrow  \bigoplus_v \mathbb{Z}/e_v\mathbb{Z}=\mathbb{Z}/|P|\mathbb{Z}$ appearing in the exact sequence  \eqref{sequence genus theory} is surjective. This gives rise to the short exact sequence 
\begin{equation}\label{sequence partial}
0\longrightarrow \CC_{K/F} \longrightarrow H^1(P, U_K)\longrightarrow \mathbb{Z}/|P|\mathbb{Z} \longrightarrow 0
\end{equation}
and the isomorphism
\begin{equation}\label{isomorphism coker}
\mathrm{coker} (Cl_F\rightarrow Cl_K^P)\cong H^2(P, U_K).
\end{equation}
Combining \eqref{sequence partial},  \eqref{isomorphism coker}, Proposition \ref{unit principal genus theorem} and  Theorem~\ref{theo-main}, we obtain claim (ii).
\end{proof}

\begin{rema}
The assertion in the proof of claim (ii) that $\mathfrak{L}^t$ is principal for some $t \in \N$ with $(t,p)=1$ can also be established by alternative methods when $k$ is $\Q$ or an imaginary quadratic field.
Namely, one may use cyclotomic units in the former case and elliptic units in the latter.
\end{rema}

Theorem \ref{theo-smallramification} yields the following group-theoretic consequence, which will be used in Lemma \ref{assump-F12}. For a profinite group $\mathcal{G}$ and an open subgroup $V$, we write $t(\mathcal{G},V)$ for the kernel of the transfer map $\mathcal{G}^{ab} \to V^{ab}$. Recall that a profinite group is called to be FAb if every open subgroup has finite abelianization.

\begin{coro}\label{prop-tk} 
Let $p$ be an odd prime. Let $\mathcal{G}$ be a FAb pro-$p$-group, and let $V$ be a normal subgroup of~$\mathcal{G}$ such that $\mathcal{G}/V \cong \Z/p^n\Z$ for some $n \in \N$.  
For each $0 \leq i \leq n$, let $V_i$ denote the unique subgroup with $V \subset V_i \subset \mathcal{G}$ and $\mathcal{G}/V_i \cong \Z/p^i\Z$.  
If $t(\mathcal{G},V) \cong \Z/p^n\Z$, then $t(V_i, V) \cong \Z/p^{n-i}\Z$ holds for every $0 \leq i \leq n$.   
\end{coro}

\begin{proof}
We may assume that $\mathcal{G}$ is finite, because we have 
$$t(\mathcal{G},V) \cong t\bigg( \frac{\mathcal{G}}{[V,V]}\,,\, \frac{V}{[V,V]}\,\,\bigg), \quad t(V_i, V) \cong t\bigg(\frac{V_i}{[V,V]}, \frac{V}{[V,V]}\bigg),$$
and the quotient $\mathcal{G}/[V,V]$ is finite. By the main theorem of \cite{HMR24}, building on earlier work of \cite{Ozaki3}, there exists a $p$-extension~$L/\Q$ whose $p$-Hilbert class field tower has Galois group isomorphic to~$\mathcal{G}$.  
Let $L_i$ denote the unramified cyclic $p$-extension of $L$ corresponding to~$V_i$.  
Then, by Artin reciprocity law, we have $\CC_{L_n/L} \cong \Z/p^n\Z$. Applying Theorem \ref{theo-smallramification} (i) (for the case $H=1$) to the unramified extension $L_n/L$ with $G_{L_n/L} \cong \mathcal{G}/V$, which we denote by $P$, one verifies that
\[
    \EE_{L_n} \cong \Z_p[P]^{r_L} \oplus \Z_p[P]/(\tr_P), 
\]
as $\Z_p[P]$-modules.
In particular, as $\Z_p[P_i]$-modules, where we put $P_i:= G_{L_n/L_i} \cong V_i/V$, we have
\[
    \EE_{L_n} \cong \Z_p[P_i]^{r_{L_i}} \oplus \Z_p[P_i]/(\tr_{P_i}).
\] 

Consequently, one infers 
\[
   t(V_i, V) \cong  \CC_{L_n/L_i} 
    \cong H^1(P_i, U_{L_n}) 
    \cong H^1(P_i, E_{L_n} ) 
    \cong \Z/p^{n-i}\Z
\]
by Theorem~\ref{theo-smallramification} (i) again.
\end{proof}

\section{Condition \eqref{condC}  without acyclicity}\label{sec-non-cohotri}

Theorem~\ref{theo-smallramification} relies on Condition~\eqref{condC}, which is typically guaranteed by $\zeta_p \notin K$ but need not hold in general (cf. Remark \ref{rema-gcu} (iii)). The purpose of this section is to establish Condition~\eqref{condC} even when $\mu_{K,p}$ is not acyclic. We first prove in \S \ref{sec-proofC} that Condition~\eqref{condC} follows from certain assumptions of Theorem \ref{theo-smallramification} together with an additional Galois descent hypothesis (Theorem \ref{theo-removingC}). We then show in \S 5.2 that this Galois descent condition is satisfied for $K/F$ when $r_k=0$ and $\zeta_p \in k$, thereby removing the descent hypothesis further. 
Consequently, in this setting, Isomorphism~\eqref{isoM} holds  (Corollary \ref{theoB-ext}).

\subsection{Verification of Condition (C) under Galois descent}\label{sec-proofC}

We retain the notation $K/k$, $P$, and $H$ from Section \ref{sec-B}. Write $|P| = p^n$ for some $n \in \N$. 
Let $F_i$ denote the unique subextension of $K/F$ with $G_{F_i/F}  \cong \Z/p^i\Z$, and let $k_i$ denote the unique subextension of $K^H/k$ with $G_{k_i/k}  \cong \Z/p^i\Z$.

\begin{theo}\label{theo-removingC}
Assume that $(\EE_K)^P = \EE_F$, and that either
\begin{itemize}
\item[(a)] $K/F$ is unramified with $\CC_{K/F} \cong \Z/|P|\Z$, or
\item[(b)] $K/F$ is ramified at a unique prime of $F$,   $p\nmid h_k$, and $\CC_{K/F} = 0$.
\end{itemize}
Then Condition~\eqref{condC}, and hence Isomorphism~\eqref{isoM}, hold in the following cases:
\begin{enumerate}
\item If $|P| = p$, 
\item If $|P| \geq p^2$ and $|\mu_{K,p}|=p$.
\end{enumerate}
\end{theo}

\begin{rema}\label{rema-gcu}
\begin{enumerate}
\item Recall that the condition $(\EE_K)^P = \EE_F$ is necessary for Isomorphism \eqref{isoM}, and is therefore indispensable. 
\item The proof of Theorem~\ref{theo-smallramification} already shows, without using Condition~\eqref{condC}, that under the assumptions of Theorem~\ref{theo-removingC}, one has $H^1(P,U_K)\cong \Z/|P|\Z$.
\item Condition~\eqref{condC} does not hold in general even when Isomorphism~\eqref{isoM} holds. For example, this is the case for any real quadratic field $K$ with $N_{K/\Q}(U_K)=1$.

\end{enumerate}    
\end{rema}

The proof of Theorem~\ref{theo-removingC} relies on Lemma~\ref{assump-F12} and Lemma~\ref{lemma-p^2}.

\begin{lemm}\label{assump-F12}
Let $K/F$ be as in Theorem \ref{theo-main}.
\begin{enumerate}
\item Under the assumptions of case~(a) of Theorem~\ref{theo-removingC}, we have $\CC_{F_i/F} \cong \Z/p^i\Z$. 
If moreover $p$ is odd, then $\CC_{K/F_i} \cong \Z/p^{n-i}\Z$ for each $0 \le i \le n$.
\item Under the assumptions of case (b) of Theorem~\ref{theo-removingC}, we have  $\CC_{F_i/F} = 0$ and  $\CC_{K/F_i}=0$ for each $i$.
\end{enumerate}
\end{lemm}

\begin{proof}

(i) By Proposition \ref{unit principal genus theorem} and \eqref{equation unramified kernel cokernel}, we have $p^i \mid |\CC_{F_i/F}|$. 
Since $\CC_{F_i/F}$ is a subgroup of the cyclic group $\CC_{K/F}$ and has exponent dividing $p^i$, it follows that $\CC_{F_i/F} \cong \Z/p^i\Z$. 
The second claim follows from Corollary \ref{prop-tk}.

(ii)   Taking $G_{F_i/F}$-invariants of the exact sequence
\begin{equation*}
    0 \longrightarrow \CC_{K/F_i} \longrightarrow Cl_{F_i} \longrightarrow Cl_K^{G_{K/F_i}},
\end{equation*}
we obtain the exact sequence
\begin{equation*}
    0 \longrightarrow \big ( \CC_{K/F_i} \big )^{G_{F_i/F}} \longrightarrow Cl^{G_{F_i/F}}_{F_i} \longrightarrow Cl^{G_{K/F}}_K.
\end{equation*}
By Theorem \ref{theo-smallramification} (ii) and its proof, we have $Cl_F \cong Cl_{F_i}^{G_{F_i/F}}$ and $Cl_F \cong Cl_K^{G_{K/F}}$, i.e., $\CC_{F_i/F} = 0,\;$  $Cl_{F_i}^{G_{F_i/F}} \cong  Cl_K^{G_{K/F}}$ and hence $\big ( \CC_{K/F_i}\big)^{G_{F_i/F}}=0$.
Since both $\CC_{K/F_i}$ and  $G_{F_i/F}$ are $p$-groups, we conclude that   $\CC_{K/F_i}=0$.
\end{proof}

\begin{rema}\label{rema-modifiedgenus}
   The second claim of Lemma \ref{assump-F12} (i)  can also be proved by a variant of genus theory developed in the forthcoming work \cite{LeeWan}.
    Using this, one obtains $H^2(G_{K/F_i}, U_K) \cong H^1(G_{K/F_i}, V_K)$ where $V_K$ denotes the closed subgroup generated by inertia subgroups at all primes of $K$ in   $(G_{\bar{K}/K})^{\mathrm{ab}}$.
    Then one can prove $H^1(G_{K/F_i}, V_K) = 0$ by the Hochschild-Serre spectral sequence.
    This method does not depend on Theorem~\ref{theo-smallramification}, in particular, it applies when $p=2$. See \cite{LeeWan} for further details.
 \end{rema}

Let $\Gamma \cong \Z/p^2\Z$, and fix a generator $\gamma$ of $\Gamma$. We have the following observation from the classification of indecomposable $\Z_p[\Gamma]$-lattices \cite{BermanGudivok, HelRei1}. Let $\Z_p[\zeta_p]$ and $\Z_p[\zeta_{p^2}]$ denote the indecomposable $\Z_p[\Gamma]$-lattices on which $\gamma$ acts by multiplication by $\zeta_p$ and $\zeta_{p^2}$, respectively. We denote by $\Gamma'$ the subgroup of $\Gamma$ of order $p$.

\begin{lemm}\label{lemma-p^2}
Let $M$ be a $\Z_p[\Gamma]$-lattice such that $\Q_p \otimes_{\Z_p} M \cong \Q_p[\Gamma]^s \oplus \Q_p[\Gamma]/(\tr_{\Gamma})$ for some $s \geq 0$, and assume $|H^1(\Gamma, M)|=p^2$ and $|H^1(\Gamma', M)| =p$. Then we have $M \cong \Z_p[\Gamma]^s \oplus \Z_p[\Gamma]/(\tr_{\Gamma}).$
\end{lemm}

\begin{proof}
According to Proposition~\ref{prop-repre}, it suffices to show that $H^1(\Gamma, M) \cong \Z/p^2\Z$. Suppose, for a contradiction, that $H^1(\Gamma, M) \cong (\Z/p\Z)^2$. By the classification of indecomposable $\Z_p[\Gamma]$-lattices (for instance using Yakovlev diagrams; see \cite[Thm. 5]{Yakovlev2}), one deduces that $$M \cong \Z_p[\zeta_p] \oplus M'$$ for some $\Z_p[\Gamma]$-lattice $M'$. For this, one may note that $\Z_p[\zeta_p]$ is the unique indecomposable $\Z_p[\Gamma]$-lattice
$N$ such that $H^1(\Gamma',N)=0$ and $H^1(\Gamma,N)\cong \Z/p\Z$. It then follows that
\begin{equation*}
H^1(\Gamma', M') \cong \Z/p\Z
\qquad\text{and}\qquad
H^1(\Gamma, M') \cong \Z/p\Z.
\end{equation*}
Let $M''$ denote the direct sum of all the indecomposable
summands in the Krull-Schmidt decomposition of $M'$ other than $\Z_p$, $\Z_p[\Gamma/\Gamma']$, and $\Z_p[\Gamma]$. Then $M''$ can occur in exactly one of the following three forms (cf. \cite[p.404]{CVM}):

\begin{enumerate}
\item The lattice $M''$ is a non-split extension of $\Z_p[\Gamma/\Gamma']$ by $\Z_p[\zeta_{p^2}]$. In this case, we have $\Q_p \otimes_{\Z_p} M''\cong \Q_p[\Gamma]$.

\item The lattice $M''$ is an indecomposable extension of $\Z_p \oplus \Z_p[\zeta_p]$ by $\Z_p[\zeta_{p^2}]$. Its rational representation is $\Q_p[\Gamma]$.

\item The lattice $M''$ is the direct sum of $\Z_p[\zeta_p]$ and an indecomposable extension of $\Z_p \oplus \Z_p[\Gamma/\Gamma']$ by $\Z_p[\zeta_{p^2}]$. In this case, we have $\Q_p \otimes_{\Z_p} M'' \cong \Q_p[\Gamma] \oplus \Q_p \oplus \Q_p(\zeta_p)$.
\end{enumerate}

Each of these possibilities leads to a contradiction. Indeed, in each case the multiplicity of $\mathbb{Q}_p$ in $\mathbb{Q}_p \otimes_{\mathbb{Z}_p} M''$ is at least that of $\mathbb{Q}_p(\zeta_{p^2})$, whereas in $\mathbb{Q}_p \otimes_{\mathbb{Z}_p} M$ the multiplicity of $\mathbb{Q}_p$ is exactly one less than that of $\mathbb{Q}_p(\zeta_{p^2})$.
\end{proof}

\begin{proof}[\textbf{Proof of Theorem \ref{theo-removingC}}]

 We recall that $H^i(P, \mu_{K})= H^i(P, \mu_{K,p})$,  $H^i(P, U_K)=  H^i(P, \Z_p \otimes U_K)$ and $H^i(P, E_K)=  H^i(P, \EE_K)$ by Lemma~\ref{lemma tensor cohomology}.

 \textbf{Case (i):} suppose $|P|=p$.
Then  we have $H^2(P, U_K)=0$ and $H^1(P, U_K)=\mb Z/p\mb Z$ by Remark~\ref{rema-gcu} (ii) and Proposition \ref{unit principal genus theorem}.
Moreover, by Isomorphism \eqref{remark unit lattice galois fixed}, there is an exact sequence 
\begin{equation}\label{exact sequence-C}
0 \to H^1(P, \mu_{K,p}) \longrightarrow H^1(P, \Z_p \otimes U_K) \longrightarrow H^1(P,\EE_K) \longrightarrow H^2(P, \mu_{K, p}) \to 0.    
\end{equation}
Since $ |H^1(P, \mu_{K,p})| =| H^2(P, \mu_{K,p})|$,  we have $H^1(P, E_K) \cong H^1(P,U_K) \cong \Z/p\Z$. 

 \textbf{Case (ii):} suppose $|P|=p^n \geq p^2$ and $|\mu_{K,p}|=p$.
 Let us denote the Galois groups $G_{F_2/F}$ and $G_{F_2/F_1}$ by $ {\Gamma}$ and $ {\Gamma}'$, respectively.
 In this case, by Remark~\ref{rema-gcu} (ii), Lemma~\ref{assump-F12} and Proposition~\ref{unit principal genus theorem}, we have
\begin{align}\label{cohomology preparation}
   H^1(P, U_{K})\cong\mb Z/|P| \mb Z,\ \  H^1(\Gamma, U_{F_2})\cong\mb Z/p^2 \mb Z,\ \   H^1( {\Gamma}', U_{F_2}) \cong \Z/p\Z, \\  
   \text{ and } \  H^2(P, U_{K})=  H^2( {\Gamma}, U_{F_2}) = H^2(\Gamma', U_{F_2})=0 \nonumber.  
\end{align}
(For case (a) and $p=2$, see Remark \ref{rema-modifiedgenus}.)
Furthermore,  $(\EE_K)^P = \EE_F$ implies $(\EE_{F_2})^{{\Gamma}} = \EE_F$.
Thus, using the $2$-periodicity of cohomology of finite cyclic groups,   by Isomorphism \eqref{remark unit lattice galois fixed}, we have the  following  commutative diagram
\begin{equation}\label{sequence level 2}
\begin{tikzcd}[column sep=1.1em, row sep=2.0em]
0 \ar{r} & \widehat{H}^{-1}(P, \mu_{K, p}) \ar{r}  \ar{d}{def^{-1}_{\mu}}& \widehat{H}^{-1}(P, \Z_p \otimes U_K) \ar{d}{def^{-1}_{U}} \ar{r}{\lambda} & \widehat{H}^{-1}(P, \EE_K) \ar{r} \ar{d}{def^{-1}_{\EE}} & \widehat{H}^0(P, \mu_{K, p}) \ar{r}  \ar{d}{def^{0}_{\mu}} & 0  \\
0 \ar{r} & \widehat{H}^{-1}({\Gamma}, \mu_{F_2, p}) \ar{r}{\rho}& \widehat{H}^{-1}({\Gamma}, \Z_p \otimes U_{F_2})   \ar{r}{\lambda'}& \widehat{H}^{-1}({\Gamma},\EE_{F_2})\ar{r} & \widehat{H}^0({\Gamma}, \mu_{F_2, p})    \ar{r} &0
\end{tikzcd}   
\end{equation}
with exact rows.
Here  $def^{-1}_{\ast}$ and $def^0_{\mu}$ are the natural maps, which are induced by the norm map from $K$ to $F_2$ and the identity map $\mu_{F, p}=(\mu_{K,p})^P \rightarrow  (\mu_{F_2,p})^{\Gamma}=\mu_{F, p}$, respectively (cf. \cite[  Ch.1, \S 9]{NSW}).

By the assumption $|\mu_{K,p}| = p$, we have 
$$|\widehat{H}^{-1}(P, \mu_{K, p})|=|\widehat{H}^{0}(P, \mu_{K, p})|=|\widehat{H}^{-1}(\Gamma, \mu_{F_2, p})|=|\widehat{H}^{0}(\Gamma, \mu_{F_2, p})|=p$$
and that  $def_\mu^0$ is an isomorphism.
Hence,  from  \eqref{cohomology preparation} and   \eqref{sequence level 2}, it follows that $  {H}^{1}(P,E_K)$ $ \cong \widehat{H}^{-1}(P,E_K) $ is isomorphic
either to $\Z/p^{n-1}\Z \times \Z/p\Z$ or to $\Z/p^{n}\Z$.

To obtain a contradiction, suppose  $\widehat{H}^{-1}(P,E_K)\cong \widehat{H}^{-1}(P,\EE_K)\cong \mb Z/p^{n-1}\mb Z\times \mb Z/p\mb Z$. 
Then there exists $x\in \widehat{H}^{-1}(P,\EE_K)$, not in the image of
$\lambda$, such that $px=0$.
A diagram chase, together with the fact that $def_\mu^0$ is an isomorphism,
shows that $def_{\EE}^{-1}(x)$ does not lie in the image of
$\lambda'$.
Since $\coker(\rho)
\cong \mb Z/p\mb Z$ and $|\widehat{H}^{-1}(\Gamma,\EE_{F_2})|=p^2$, it follows that
$\widehat{H}^{-1}(\Gamma,\EE_{F_2})\cong \mb Z/p\mb Z\times \mb Z/p\mb Z$.
It therefore suffices to show that $|H^1(\Gamma',\EE_{F_2})|=p$,
as this would contradict Lemma~\ref{lemma-p^2}  and Theorem~\ref{theo-main}.

If not, we have $|H^1(\Gamma',\EE_{F_2})|=p^2$ or $|H^1(\Gamma',\EE_{F_2})|=1$, since $H^1( {\Gamma}', U_{F_2}) \cong \Z/p\Z$ by \eqref{cohomology preparation},   and the orders $|H^1( {\Gamma}',U_{F_2})|$ and $|H^1( {\Gamma}',\EE_{F_2})|$ differ by at most $p$.
We show that  neither case is possible.

First suppose  $|H^1(\Gamma',\EE_{F_2})|=1$.
Then we have the following exact sequence
\begin{equation*} 
 0=H^1(\Gamma',\EE_{F_2}) \longrightarrow H^2(\Gamma', \mu_{F_2, p}) \to  H^2(\Gamma', \Z_p \otimes U_{F_2})=0 
\end{equation*}
by \eqref{cohomology preparation},
and this yields $H^2(\Gamma', \mu_{F_2, p})=0$.
This leads to a contradiction since $H^2(\Gamma', \mu_{F_2, p})\cong  \widehat{H}^0(\Gamma', \mu_{F_2, p})\cong \mb Z/p\mb Z$.

Now suppose $|H^1(\Gamma',\EE_{F_2})|=p^2$.
We claim that the restriction map $res_{\EE}: H^1(\Gamma, \EE_{F_2}) \to H^1(\Gamma', \EE_{F_2})$ is surjective by applying the four lemma to the diagram
\begin{equation*}
\begin{tikzcd}
{H^1(\Gamma, \Z_p \otimes U_{F_2})} \arrow[d, "res_U"] \arrow[r]  & {H^1(\Gamma, \EE_{F_2})} \arrow[d, "res_{\EE}"] \arrow[r]   & {H^2(\Gamma,\mu_{F_2,p})} \arrow[d, "res_{\mu}"] \arrow[r] & 0 \arrow[d] \\
{H^1(\Gamma',  \Z_p \otimes U_{F_2})} \arrow[r]                   & {H^1(\Gamma', \EE_{F_2})} \arrow[r]                   & {H^2(\Gamma', \mu_{F_2,p})} \arrow[r]                & 0          
\end{tikzcd}
\end{equation*}
with exact rows, where $res_\ast$ denotes an appropriate restriction map.

The surjectivity of $res_{\mu}$ can be checked directly from the assumption on the orders of $\mu_{F_2,p}$ and $\mu_{F,p}$. By genus theory, $res_U$ corresponds to the lifting map $\CC_{F_2/F} \to \CC_{F_2/F_1}$ in case (a) and to the natural map $I_{F_2}^{\Gamma}/I_F \to I^{\Gamma'}_{F_2}/I_{F_1}$ in case (b), respectively (cf. \cite[Prop. 3.5]{KumonLim}). 
Hence, $res_U$ is surjective. 
As a consequence, if $H^1(\Gamma', \EE_{F_2})$ has $p^2$ elements, then $res_{\EE}$ is an isomorphism, i.e., $ H^1( {\Gamma}/ {\Gamma}', \EE_{F_2}^{ {\Gamma}'})= \mathrm{ker} (res_{\EE})=0$.

However,   $ H^1( {\Gamma}/ {\Gamma}', \EE_{F_2}^{ {\Gamma}'})$ is non-trivial.
Indeed, the Krull-Schmidt decomposition of $\EE_{F_2}^{ {\Gamma}'}$ as a $\Z_p[ {\Gamma}/ {\Gamma}']$-lattice has a direct summand isomorphic to $\Z_p[ {\Gamma}/ {\Gamma}']/(\tr_{\Gamma/\Gamma'})$.
This follows from the isomorphisms $$\Q_p \otimes_{\Z_p} \EE_{F_2}^{ {\Gamma}'} \cong (\Q_p \otimes_{\Z_p} \EE_{F_2})^{ {\Gamma}'} \cong \Q_p[ {\Gamma}/ {\Gamma}']^{r_F} \oplus \Q_p[ {\Gamma}/ {\Gamma}']/(\tr_{\Gamma/\Gamma'}),$$ together with Diederichsen's classification of indecomposable $\Z_p[\Gamma/\Gamma']$-lattices (cf. \cite[\S 2]{HelRei1}).
Note that $ {H}^{1}(\Gamma/\Gamma', \mb Z_p[\Gamma/\Gamma']/(\tr_{\Gamma/\Gamma'}))  \cong \mb Z/ p  \mb Z$.
This contradiction shows   $|H^1(\Gamma',\EE_{F_2})| \neq p^2$, and hence we obtain $|H^1(\Gamma',\EE_{F_2})|=p$.

Putting everything together, we conclude that
$  {H}^{1}(P,E_K) \cong \Z/p^{n}\Z \cong  {H}^{1}(P,U_K)$.
\end{proof}

\subsection{Galois descent in special cases}

In general, the Galois descent $(\EE_K)^P = \EE_F$ is difficult to analyze. Nevertheless, we record the following observation.

\begin{prop}\label{prop-C}
Let $K/k$ be a finite Galois extension of number fields. 
Assume $p\nmid h_k$, $r_k=0$, and $\zeta_p \in k$.
Furthermore, assume that $G_{K/k} = H \times P$, where $P$ is a cyclic $p$-group and $H$ is a group of order prime to $p$. 
Then $(\EE_K)^P = \EE_F$, where $F=K^P$. 
\end{prop}
 Combining Proposition~\ref{prop-C} with  Theorem~\ref{theo-removingC}, we obtain the following corollary.
 
\begin{coro}\label{theoB-ext}
The equality  $(\EE_K)^P = \EE_F$ holds in the following cases:
\begin{enumerate}
\item  $p=2$, $k=\Q$, and $K$ is totally real.
\item   $p=2$, $k$ is an imaginary quadratic field with $2\nmid h_k$ and $\zeta_4 \notin K$.
\item   $p=3$, $k=\Q(\sqrt{-3})$, and $\zeta_9 \notin K$.
\end{enumerate}
Under any of these conditions,  assume further that:
\begin{itemize}
\item[(a)] $K/F$ is unramified with $\CC_{K/F} \cong \Z/|P|\Z$, or
\item[(b)]  $K/F$ is ramified at a unique prime of $F$,   $p\nmid h_k$, and $\CC_{K/F} = 0$.
\end{itemize}
Then Isomorphism~\eqref{isoM} holds for $K/k$.
\end{coro}

The following proposition complements Lemma~\ref{coho-trivial} (ii) (a) by showing that the Galois descent property still holds for $p=2$ even without cohomological triviality. We use it in the proof of Proposition~\ref{prop-C}.

\begin{prop}\label{proposition unit lattice fixed 1}
Let $M/L$ be a cyclic $p$-extension of number fields. 
If $p=2$, we further assume that there is a prime of $L$ above $2$  whose ramification index over $L \cap \Q(\zeta_{2^\infty})$ is odd.
Then $(\EE_M)^{G_{M/L}} = \EE_L$ in each of the following cases:
\begin{enumerate}
\item $M=L(\mu_{M,p})$;
\item $M/L$ is totally ramified at a non-$p$-adic prime of $L$.
\end{enumerate}
\end{prop}

\begin{proof}
By Lemma \ref{lemma tensor cohomology}, one has $(\EE_M)^{G_{M/L}}=\EE_L$ if and only if $(E_M)^{G_{M/L}}=E_L$.
\vskip 5pt

\noindent \textbf{$\bullet$ Case (i) :} 
By Lemma \ref{coho-trivial}, the claim is non-vacuous only when $p=2$ and $L \cap \Q(\mu_{M,2})$ is totally real. Since $M/L$ is cyclic, Lemma \ref{coho-trivial} (ii)(b) implies $M = L(\sqrt{-1})$. As $\sqrt{-1} \in M$, we have 
\[
M\cap \Q(\zeta_{2^\infty})= \Q(\mu_{M,2})=\Q(\zeta_{2^{m+2}}),
\]
for some $m \geq 0$, which is the $m$-th layer of the cyclotomic $\Z_2$-extension of $\Q(\sqrt{-1})$. Since $L \cap \Q(\mu_{M,2})$ is totally real, it follows that
\begin{equation}\label{equation root of unity field}
    L \cap \Q(\mu_{M,2}) = L\cap \big (M \cap \Q(\zeta_{2^\infty}) \big) = \mathbb{Q}(\zeta_{2^{m+2}})^{G_{M/L}} =\Q(\eta_m),
\end{equation} 
where $\eta_m$ denotes $\zeta_{2^{m+2}}+\zeta_{2^{m+2}}^{-1}$.

  Suppose that there exists $\epsilon \in U_M$ whose class $\bar{\epsilon}$ lies in $(E_M)^{G_{M/L}} \setminus E_L$.
  The quotient group $(E_M)^{G_{M/L}}/E_L$ has exponent $2$ by $\eqref{remark unit lattice galois fixed}$. Hence, after replacing $\epsilon$ by a suitable odd power if necessary, we obtain
\begin{equation}\label{eq-non-p-power}
    \epsilon^2 = x \zeta \qquad \text{for some} \quad x \in U_L, \quad \zeta \in \mu_{M,2}.
\end{equation}
We claim the following : $\zeta$ generates $\mu_{M,2}$; the extension $L(\sqrt{\zeta})/L$ is a biquadratic Galois; and $L(\sqrt{\zeta})$ contains $L(\sqrt{x}) \neq M$. In particular, we have
\begin{equation}\label{iso-Z2Z2}
G_{L(\sqrt{\zeta})/L}  \cong G_{\Q(\zeta_{2^{m+3}})/\Q(\eta_m)}  \cong (\Z/2\Z)^2.
\end{equation}
Since $L(\sqrt{x}) \neq M=L(\sqrt{-1})$, it follows that $L(\sqrt{x}) = L(\sqrt{\pm (2+\eta_m)})$, and
\begin{equation}\label{Kummercongruence}
    x \equiv \pm (2+\eta_m) \pmod{L^{\times 2}}, 
\end{equation}
where we note $\eta_{m+1}^2=2+\eta_m$.
\vskip 5pt
Now we prove the above assertions. Since $E_M$ is torsion-free, we have $x \not\in U_L^2$, and hence $L(\sqrt{x})/L$ is quadratic.  
  We also have $M \neq L(\sqrt{x})$. Otherwise, we would have $x \equiv -1 \pmod{L^{\times 2}}$, which implies $\bar{\epsilon} \in E_L$, as $\bar{\epsilon}^2 = \bar{x} \in E_L^2$.

We claim that $\sqrt{\zeta} \notin M$. Indeed, otherwise we would have $\sqrt{x} \in M$. Since $\epsilon = \pm\sqrt{x\zeta} \in M$, this yields a contradiction. Hence, $\zeta$ generates $\mu_{M,2}$.

Consequently, we have $[L(\sqrt{\zeta}):L]=4$ and \eqref{iso-Z2Z2} follows from \eqref{equation root of unity field}. Together with $M \neq L(\sqrt{x})$, this gives $L(\sqrt{x})=L(\sqrt{\pm (2+ \eta_m)})$ and \eqref{Kummercongruence} follows.

Let $\mf q $ be a prime of $L$ dividing $2$ such that  the ramification index of $\mf q$ over $  \Q(\eta_m)$ is odd. Let $v$ be the normalized  valuation of $L$ corresponding to $\mf q$. In particular, $v(2+\eta_m)$ is equal to the ramification index of $L$ over $\Q(\eta_m)$ at $\mf q$ since $2+\eta_m$ is a uniformizer at the unique $2$-adic place $\mf q \cap \mb Q(\eta_m)$  of 
   $   \Q(\eta_m)$. 
   Note also $v(2+\eta_m) = v(\eta_m)$ when $m>0$.
   By the assumption on the ramification condition on $L / \Q(\eta_m)$, we obtain that     $v(2+\eta_m)$ is odd.
   It follows that
\begin{equation*}\label{arg-parity}
 0=v(x)\equiv v(\pm (2+\eta_m)) \equiv 1\pmod{2}.   
\end{equation*}
 This contradiction settles case (i).

\vskip 10pt
\noindent \textbf{$\bullet$ Case (ii) :}  Suppose there exists $\epsilon \in U_M$ representing a class in $(E_M)^{G_{M/L}} \setminus E_L$. 
Let $p^b$ denote the order of the class of $\bar{\epsilon}$ in the quotient $(E_M)^{G_{M/L}}/E_L$. 
After replacing $\epsilon$ by a suitable power prime to $p$, we may assume, by case (i) applied to $L(\mu_{M,p})/L$, that $p^b$ is the minimal integer such that $\epsilon^{p^b} \in L(\mu_{M,p})$.
It then follows that $L(\mu_{M,p}, \  \epsilon^{p^{b-1}})/L(\mu_{M,p})$ is a degree-$p$ extension unramified outside $p$. This yields a contradiction.
\end{proof}

\begin{proof}[\textbf{Proof of Proposition \ref{prop-C}}]
We treat separately the cases $\mu_{k,p}=\mu_{K,p}$ and $\mu_{k,p} \neq \mu_{K,p}$.

\vskip 5pt

\noindent \textbf{$\bullet$ Case $\mu_{k,p}=\mu_{K,p}$: }
Suppose that there exists $\epsilon \in U_K$ whose class $\bar{\epsilon} \in E_K$ lies in $E_K^P \setminus E_F$. By the same argument for \eqref{eq-non-p-power}, there exists a minimal positive integer $b$ such that $\epsilon^{p^b} =x$ for $x \in U_F$.
Since $G_{K/k} = H \times P$ and $\zeta_p\in k$, there exists a Kummer extension $L/k$ of degree $p$ such that $F(\sqrt[p]{x})=FL$. It follows that $x \equiv y \pmod{F^{\times p}}$ for some $y \in k^{\times}$. 

Since $x$ is in $U_{F}$, we have $v(y)\equiv 0 \pmod{p}$ for every   normalized discrete valuation $v$ of $F$.
Since $[F:k]$ is prime to $p$, we deduce that the principal ideal $(y)$ is a  $p$-th power as a fractional  ideal of $k$.
Since $p \nmid h_k$ and $r_k=0$, it follows that $y=y'^p \zeta$ for some $y' \in k^{\times}$ and $\zeta \in    \mu_{k}=U_k$.

From $x \equiv \zeta \pmod{F^{\times \, p}}$, we obtain $\sqrt[p]{\zeta} \in F(\sqrt[p]{x}) \subseteq K$. By the assumption $\mu_{K,p} = \mu_{k,p}$, we have $\zeta \in k^{\times p}$. Hence, we have $\epsilon^{p^b} = x \in F^{\times p}$, which contradicts the minimality of $b$.

\vskip 10pt
\noindent \textbf{$\bullet$ Case $\mu_{k,p} \neq \mu_{K,p}$ :}  
First we claim that 
\begin{equation}\label{equation root of 1 gal descent}
\big (E_{F(\mu_{K,p})} \big )^{G_{F(\mu_{K,p})/F}} = E_F.
\end{equation}
For proving this claim, we treat three cases.
    \begin{itemize}
        \item 
     $p\neq 2$: the claim \eqref{equation root of 1 gal descent} follows from Lemma~\ref{coho-trivial}.
        \item 
 $p=2$  and $k=\mathbb{Q}$: the assumption that $[F:\mathbb{Q}]$ is odd allows us to apply Proposition~\ref{proposition unit lattice fixed 1} (i).
        \item 
          $p=2$  and $k$ is imaginary quadratic:
           \begin{itemize}
               \item 
                    If $k \subset  \Q(\zeta_{2^\infty })$,  we can use Lemma~\ref{coho-trivial} again.
               \item 
                  
   If $k \cap \Q(\zeta_{2^\infty }) =\Q$,  then $k/\Q$ is unramified at $2$.
   Indeed, if $k/\Q$ is ramified at $2$ and   $k \cap \Q(\zeta_{2^\infty }) =\Q$, then $k=\Q(\sqrt{-m})$ for some square-free nonnegative integer $m$,  which has an odd prime factor.
     Then, by the classical Gauss' genus theory, $h_k$  is divisible by $2$.
     This contradicts to the assumption $2\nmid h_k$.
 Then  we can use Proposition~\ref{proposition unit lattice fixed 1} (i) since all ramification indices of  $F/\mb Q$ at $2$  are odd.
           \end{itemize}
    \end{itemize}

Now suppose that there exists $\epsilon \in U_K$ whose class $\bar{\epsilon}$ lies in $(E_K)^P \setminus E_F$.
Let $b > 0$ be minimal with $(\bar{\epsilon})^{p^b} \in E_F$. 
 Replacing $\epsilon$ by a suitable power   if necessary, we may assume that $\epsilon^{p^b}=x\zeta $ for some $x \in U_F$ and $\zeta \in \mu_{K,p}$.
  Moreover, the integer $p^b$ is the minimal number such that $\epsilon^{p^b} \in F(\mu_{K,p})$ by \eqref{equation root of 1 gal descent}.
            
 As $P$ is cyclic, the minimality of $b$ implies both $F(\epsilon^{p^b})=F(\mu_{K,p})$ and $[F(\epsilon^{p^{b-1}}):F(\epsilon^{p^b})]=p$.

Since $F(\epsilon^{p^{b-1}})=F( \sqrt[p]{ \zeta x}   )$, we obtain $ \sqrt[p]{  x}  \in F(   \sqrt[p]{ \zeta } ,\ \epsilon ^{p^{b-1}})$.
Let $M_1/k$ be the extension of $k$ in $K^H$ such that $F(\epsilon^{p^{b-1}})=FM_1$. 
One checks that the abelian extension $M_1(  \sqrt[p]{ \zeta } )/k$ satisfies $F \cap M_1(   \sqrt[p]{ \zeta } )=k$ and $FM_1(   \sqrt[p]{ \zeta } )=F( \sqrt[p]{ \zeta }, \  \sqrt[p]{  x}  )$.

Consequently, there exists an extension $M_2/k$ of degree $\le p$ in $M_1(   \sqrt[p]{ \zeta } )/k$ such that $M_2F=F(\sqrt[p]{x})$. 
By Kummer theory, there exists $y \in k^{\times}$ such that $x \equiv y \pmod{F^{\times p}}$.
By the same argument as in the case $\mu_{k,p} = \mu_{K,p}$, one checks $y=  y '^p  \zeta'$ for some $y' \in k^{\times}$ and $\zeta' \in    \mu_k=U_k$. 
By replacing $\epsilon$ by a suitable power again, we may assume $\zeta'\in \mu_{k, p}$.
By the assumption $\mu_{K,p} \neq \mu_{k,p}$, we have $ \sqrt[p]{ \zeta' }  \in \mu_{K,p}$ and hence 
$$
F(\sqrt[p]{x})=F(\sqrt[p]{y})=F(\sqrt[p]{\zeta'}) \subseteq F(\mu_{K,p}) = F(\epsilon^{p^b}).
$$
In particular,  $x \in U_{F(\mu_{K,p})}^p$ and this implies $F(\epsilon^{p^{b-1}})  =  F(   \sqrt[p]{ \zeta } ) =F( \sqrt[p]{ \zeta } , \  \sqrt[p]{x}) $ since $\epsilon^{p^b}=x\zeta$.
Noting that $F(\epsilon^{p^{b-1}})/F(\epsilon^{p^{b}})$ is of degree $p$,    we obtain $   \sqrt[p]{ \zeta }  \notin \mu_{K, p}$. This is a contradiction since  $   \sqrt[p]{ \zeta } \in F(\epsilon^{p^{b-1}})\subset K$, i.e., $   \sqrt[p]{ \zeta } \in \mu_{K, p}$.
\end{proof}

\section{Galois structure of units in a \texorpdfstring{$\Z_p$}{  $\Z_p$}-extension of a number field}

Let $F/k$ be a  finite Galois extension where the infinite places are unramified.
Let $p$ be a prime with $p\nmid h_k\cdot [F:k]$.
Let $k_{\infty}/k$ be a $\Z_p$-extension such that $F_{\infty}/F$ is ramified at a unique prime of $F$, where $F_\infty=Fk_{\infty}$.

For each $n \in \N$, denote by $k_n$ and $F_n$ the $n$-th layers of $k_{\infty}/k$ and $F_{\infty}/F$, respectively. Following common notation in literatures on Iwasawa theory, we write
\begin{equation*}
    A_n:=Cl_{F_n}[p^{\infty}], \quad \Gamma_n:=G_{F_n/F}.
\end{equation*}
Since $G_{F_n/k}$ is the direct product of $\Gamma_n$ and 
$G_{F_n/k_n} \cong G_{F/k}$, Theorem~\ref{theo-smallramification} applies.

\vskip 5pt

  It is an observation, made precise in \cite[Cor.~4.10]{All},   that Vandiver conjecture for $p$ is equivalent to Greenberg conjecture (\cite{G}) for $\Q(\zeta_p)^+$ and the existence of local Minkowski units at $p$ in each layer of the cyclotomic $\Z_p$-extension of $\Q(\zeta_p)^+$. We record below a variant of this equivalence that extends \cite[Cor. 4.10]{All} to the present, more general setting.

\begin{prop}\label{prop-GVM} Let $F/k$ and $p$ be as above and assume moreover that Condition \eqref{condC} holds for each $F_n/k$. Then the following are equivalent:
\begin{enumerate}
    \item Isomorphism \eqref{isoM} holds for $F_n/k$ for each $n$, and $|A_n|$ is bounded independently of $n$.
     \item Isomorphism \eqref{isoM} holds for $F_n/k$ for sufficiently large $n$, and $|A_n|$ is bounded independently of $n$.
    \item $A_n=0$ for each $n$.
    \item $A_0=0$.
\end{enumerate}
\end{prop}

\begin{proof}
The equivalence of (i) and (ii) follows from Lemma~\ref{remark subfield Minkowski}. The equivalence of (iii) and (iv) follows from the fact that $F_n/F$ is ramified at a unique prime. By \cite[Prop.~2]{G}, the orders $|A_n|$ are uniformly bounded if and only if $A_n$ capitulates in $F_{\infty}$ for each $n$. Combining this with Theorem~\ref{theo-smallramification}~(ii), we deduce that (i) and (iii) are equivalent.
\end{proof}

We work on $\Z_p$-extensions $F_{\infty}/F$ satisfying the above equivalent conditions (see Remark \ref{lambda1case} for examples outside this class).

\begin{rema}
\begin{enumerate}
\item For examples of $\Z_p$-extensions with uniformly bounded $|A_n|$, see \cite{KumaMizu} and the references therein.
\item By genus theory, the uniform boundedness of $|A_n^{\Gamma_n}|$ together with Isomorphism~\eqref{isoM} for $F_n/k$ for all $n$ forces the number of primes of $F$ ramifying in $F_\infty/F$ to be equal to one (cf. \cite[Cor.~4.8]{All}).

\item The uniform boundedness of   $|A_n^{\Gamma_n}|$ in $\Z_p$-extensions that are unramified outside a unique $p$-adic prime has been studied in  \cite{G}. When $r_k=0$, a version of the Leopoldt conjecture implies that the orders $|A_n^{\Gamma_n}|$ are uniformly bounded (cf. \cite[\S 13.5]{ Washington} and \cite{C}).
\end{enumerate}
\end{rema}

\begin{theo}[Theorem \ref{theo-Iwasawa}]
Let $F_{\infty}/F$, $k$, and $p$ be as in the beginning of this section. Assume that Condition~\eqref{condC} holds for each $F_n/F$ and  $p\nmid h_F$.
Then, we have isomorphisms
\begin{align*}
 & \varprojlim_{n \in \N} \, \EE_{F_n} \cong \Lambda[G_{F/k}]^{r_k+1} \ \text{ and } \\
 &   \varprojlim_{n \in \N} \, (\Z_p \otimes U_{F_n}) \cong \Big( \, \varprojlim_{n \in \N} \, \mu_{F_n, p} \, \Big )   \oplus\Lambda[G_{F/k}]^{r_k+1}
\end{align*}
of $\Lambda[G_{F/k}]$-modules, where the inverse limits are taken with respect to the norm maps and $\Lambda:=\Z_p[[G_{F_\infty/F}]]$ denotes the Iwasawa algebra of $G_{F_\infty/F}$.

\end{theo}

\begin{proof}
For simplicity of notation, write $G_n:=G_{F_n/k}$ and identify $G_{F_n/k_n}$ with $H:=G_{F/k}$ for each $n \in \N$. 
By Proposition \ref{prop-GVM}, we have $$\EE_{F_n} \cong \Z_p[G_n]/(\tr_{G_n}) \oplus \Z_p[G_n]^{r_k}$$ for all $n \in \N$. 
Let $I_n$ denote the maximal ideal of $\Z_p[\Gamma_n]$. Then, for $n \geq 1$, we have $\EE_{F_n}/I_n \EE_{F_n} \cong \F_p[H]^{r_k+1}$. 
As a consequence, the map $$\overline{N}_n: \EE_{F_{n+1}}/I_{n+1}\EE_{F_{n+1}} \to \EE_{F_n}/I_n \EE_{F_n},$$ induced by the norm map $N_{F_{n+1}/F_n}$, is an isomorphism of $\mb F_p[H] $-modules, as it is surjective by Lemma~\ref{remark subfield Minkowski}.

Let $y_n \in \EE_{F_n}$ and  $\{x_{n,i}\}_{1 \leq i \leq r_k} \subset \EE_{F_n}$ be such that
\begin{equation*}
\Z_p[G_n]y_n \cong  \Z_p[G_n]/(\tr_{G_n}),\ \ \Z_p[G_n]x_{n,i} \cong \Z_p[G_n] , \ \  \EE_{F_n} = \Z_p[G_n]y_{n}  \oplus\bigg(\bigoplus_{i=1}^{r_k}\Z_p[G_n]x_{n,i}\bigg).   
\end{equation*}
Choose  arbitrary preimages $y_{n+1}$ and $x_{n+1,i}$ in $\EE_{F_{n+1}}$ of $y_n$ and $x_{n,i}$, respectively, under the norm map $N_{F_{n+1}/F_n}$. Since $\overline{N}_n$ is an isomorphism, the classes of $y_{n+1}$ and $x_{n+1,i}$ in $\EE_{F_{n+1}}/I_{n+1}\EE_{F_{n+1}}$ generate this module over $\F_p[H]$. 
As $I_{n+1}$ is contained in the Jacobson radical of $\Z_p[G_{n+1}]$, Nakayama's lemma implies that $y_{n+1}$ and $\{x_{n+1,i}\}_{1 \leq i \leq r_k}$ generate $\EE_{F_{n+1}}$ as a $\Z_p[G_{n+1}]$-module.

 By \eqref{dia-compatib}, we have $\tr_{G_{n+1}} \,y_{n+1}=0$. 
Comparing $\Z_p$-ranks, we obtain 
$$
 \Z_p[G_{n+1}]y_{n+1} \cong  \Z_p[G_{n+1}]/(\tr_{G_{n+1}}), \ \  \Z_p[G_{n+1}]x_{n+1,i} \cong \Z_p[G_{n+1}],
 $$
 $$
\hbox{and } \   \EE_{F_{n+1}}  = \Z_p[G_{n+1}]y_{n+1} \oplus\bigg(\bigoplus_{i=1}^{r_k}\Z_p[G_{n+1}]x_{n+1,i}\bigg).
 $$
 
Hence, we can   choose $y_1, x_{1,i} \in \EE_{F_1}$, and then, by induction, construct compatible systems $\{x_{n,i}\}_n$ and $\{y_n\}_n$ such that
\[
N_{F_{n+1}/F_n}(x_{n+1,i})=x_{n,i}
\quad \text{and} \quad
N_{F_{n+1}/F_n}(y_{n+1})=y_n
\]
for each $n$. 
This yields the following commutative diagram of $\Z_p[G_{n+1}]$-modules with exact rows:

\begin{equation*}
\begin{tikzcd}
0 \arrow[r] & \Z_p \arrow[r, "\alpha_{n+1}"] \arrow[d, "p"]   & {\Z_p[G_{n+1}] \oplus \Z_p[G_{n+1}]^{r_k}} \arrow[r, "\beta_{n+1}"] \arrow[d]   & \EE_{F_{n+1}} \arrow[r] \arrow[d, "N_{F_{n+1}/F_n}"] & 0 \\
0 \arrow[r] & \Z_p \arrow[r, "\alpha_{n}"]                    & {\Z_p[G_{n}] \oplus \Z_p[G_{n}]^{r_k}} \arrow[r, "\beta_{n}"]                     & \EE_{F_{n}} \arrow[r]                                & 0
\end{tikzcd}    
\end{equation*}
where $\alpha_{n}$ is defined by $\alpha_{n}(1) = (\tr_{G_{n}},0)$,  and $\beta_{n}$ is the $\Z_p[G_{n}]$-linear map sending the standard basis vectors to
\[
(1,0) \mapsto y_{n}, \  \text{ and } \ (0,e_{n,i}) \mapsto x_{n,i} \ (1 \le i \le r_k).
\]
Here, $\{e_{n,i}\}$ denotes the standard basis of $\Z_p[G_{n}]^{r_k}$.

Since all modules in the diagram are compact, the first isomorphism is obtained by taking the inverse limit over $n$,  yielding
$$
\varprojlim_n \EE_{F_n} \cong  \varprojlim_n    \Z_p[G_n]^{r_k+1} \cong \Lambda[G_{F/k}]^{r_k+1} .
$$
The second isomorphism now follows by taking the inverse limits of the exact sequences
$$
0 \longrightarrow \mu_{F_n, p}  \longrightarrow  \Z_p \otimes U_{F_n}  \longrightarrow  \EE_{F_n}  \longrightarrow   0 ,
$$
and using the freeness of $\varprojlim_n \EE_{F_n}$ as a $\Lambda[G_{F/k}]$-module.
\end{proof}

Theorem~\ref{theo-Iwasawa} applies, for instance, to the cyclotomic $\Z_p$-extension of $F$ when $p$ is odd and $F/\Q$ is a totally real Galois extension of degree prime to $p$ with a unique $p$-adic prime and $p$-class number $1$. We refer to \S \ref{sec-numexam} for numerical examples.

\vskip 10pt

Furthermore, if $F = k$ and $r_k = 0$, then $(\EE_{k_n})^{\Gamma_n} = \EE_k = 0$. 
Hence, by Theorem~\ref{theo-removingC}, one may dispense with Condition~(C) under the conditions $\mu_{k,p}=\mu_{k_n,p}$ for all $n \in \N$ and $\zeta_p \in k$, and obtain the following corollary.

\begin{coro}\label{coro-IwaModul}
Theorem~\ref{theo-Iwasawa} applies to the following cases:
\begin{enumerate}
\item The cyclotomic $\Z_2$-extension of $\Q$;
\item Any $\Z_3$-extension of $\Q(\sqrt{-3})$ not containing $\zeta_9$;
\item Any $\Z_2$-extension of an imaginary quadratic field $k$ not containing $\zeta_4$, where $2$ does not split in $k$ and $2\nmid h_k$;
\item An elliptic $\Z_2$-extension of an imaginary quadratic field $k$, where $2$ splits in $k$ and $2\nmid h_k$.
\end{enumerate}
\end{coro}

\begin{rema}\label{remark unit Iwasawa module history}
\begin{enumerate}
\item The Iwasawa module structure of the inverse limit of unit lattices appears frequently in the literatures on tamely ramified Iwasawa modules for $\Z_p$-extensions (cf. \cite[Lem.~2.2]{TMO13}). The approach in~\cite{TMO13} differs from ours in that the authors use a compatible system of cyclotomic units in the finite layers, each of which is itself a local Minkowski unit.

\item In~\cite[Ch.~11, \S3]{NSW}, the Iwasawa module of units is studied using the theory of modules over the Iwasawa algebra. Theorem~\ref{theo-Iwasawa} is complementary to these results in that it is obtained under a different set of assumptions and does not rely on the weak Leopoldt conjecture. Related results can also be found in~\cite{Belliard}, where the $\Lambda$-freeness of the Iwasawa module of units was established by studying properties of Iwasawa modules and using the $\Lambda$-freeness of the corresponding Iwasawa module of $S$-units. Compare also the results of Nickel \cite{NickelSwan}, where the structure of the Iwasawa module arising as the inverse limit of unit groups is studied up to pseudo-isomorphism for certain $p$-adic Lie extensions.  
\end{enumerate}    
\end{rema}

\begin{rema}\label{lambda1case}
There are several $\Z_p$-extensions of number fields for which every finite layer admits a local Minkowski unit at $p$, while the $\Z_p$-extension itself has nonzero Iwasawa $\lambda$-invariant. Let $k$ be an imaginary quadratic field and $p$ an odd prime such that $p\nmid h_k$.
Assume that $p$ splits in $k$ as $\p\p'$, and that neither $\p$ nor $\p'$ splits in the maximal $p$-ramified pro-$p$-extension $k_{S_p}(p)$ of $k$. 
Let $M$ and $M'$ denote the unique $\Z_p$-extensions of $k$ unramified outside $\p$ and $\p'$, respectively.

Let $L$ be a $\Z_p$-extension of $k$ whose first layer differs from those of $M$ and $M'$ (e.g., the cyclotomic or anticyclotomic $\mb Z_p$-extension of $k$).
By the theory of central class fields (cf. \cite[Theorem~4]{Burns2}), one checks that the $n$-th layer of $L/k$ has $p$-class number $p^n$. 
Hence, the Iwasawa invariants of $L$ satisfy $\lambda=1$ and $\mu=0$ (cf. \cite{Ozaki2}). 
Moreover, by \cite[Thm. 5]{Burns2}, each layer of $L$ admits a local Minkowski unit at $p$. As an explicit example, one may take \(k=\Q(\sqrt{-11})\) and \(p=3\).

\end{rema}

\section{Local Minkowski units with many ramified primes}\label{sec-manyram}

In this section, we study cyclic extensions~$K/k$ admitting local Minkowski units at~$p$ for which arbitrarily many primes ramify in $K/F$, where $F$ is the fixed field of a Sylow $p$-subgroup of $G_{K/k}$.

Allowing more than one prime to ramify makes genus theory substantially more difficult to apply, since the arithmetic becomes harder to control (cf.~\cite[\S4]{KumonLim}). A common approach is therefore to choose the ramified primes carefully and exploit the relationship between the Galois group and the corresponding decomposition and inertia subgroups (cf.~\cite{BLM23, BouaLim}, \cite[\S4]{Burns2}).

In particular, Burns~\cite[Thm. 5]{Burns2} used the theory of central class fields~\cite{Frohlich} to construct abelian $p$-extensions admitting local Minkowski units in which exactly two primes ramify and whose Galois groups range over all abelian $p$-groups of rank $2$. This use of central class fields makes it possible to treat non-cyclic abelian $p$-extensions with prescribed ramification, while imposing restrictions on the number of ramified primes.

By contrast, for cyclic extensions, it suffices to control a single cohomology group by Theorem~\ref{theo-main}. By carefully controlling the splitting and ramification behavior, we obtain the following theorem.

\begin{theo} [Theorem \ref{theo-D}] \label{theo-many ramification}
         Let $F$ be a totally real cyclic number field and   $p$   an odd prime number.
   Assume $p\nmid h_F\cdot [F:\Q]$
   and  $F \cap \Q(\zeta_p)=\Q$. 
   Then, for any given $r\ge 0$ and  $n\ge n_1 \ge n_2 \ge \cdots \ge n_r\ge 1  $, there exist infinitely many  cyclic extensions $L/\Q$ of degree $p^{n}$ such that:
   \begin{itemize}
   \item
      $Cl_K[p^\infty] \cong Cl_L[p^\infty] \cong \prod_{i=1}^{r}\mb Z/p^{n_i}\mb Z\,$ where  $K=LF$.
      \item
       $K$ admits a local Minkowski unit at $p$.
     \item
     $K/F$ is ramified at exactly $r+1$  primes.
           \end{itemize}
\end{theo}

\begin{proof}
 
 For convenience, we write $n=b_1, n_1=b_2, \dots, n_r=b_{r+1}$ and $s=r+1$.
Inductively, using the Chebotarev density theorem, we  find rational primes $q_i $, cyclic extensions $M_i/\mb Q$  of degree $p^{b_i}$ and $L_i/ \mb Q$  of degree $p^n$ for $1\le i\le s$, as follows:
\begin{itemize}
 \item 
  First, choose $q_1$ that splits  completely in $\Q(\zeta_{p^{b_{1}}})$ and is  inert in $F$.
  Note that $\Q(\zeta_{p^{b_{1}}}   ) \cap F =\mb Q$.
Let $M_1/\mb Q$ be the unique subextension of $\Q(\zeta_{q_1})/\Q$ of degree $p^{b_1}$ and put $L_1:=M_1$. 
  Then, by Kummer theory, we have $\Q(\zeta_{p^{a}},   \sqrt[p^{a}]{q_1}) \cap FL_{1}=\mb Q$ for any $a$.
   \item 
   Now, suppose inductively that $\{q_1, \dots, q_{i-1}\}$, $\{M_1, \dots, M_{i-1}\}$, and $\{L_1, \dots, L_{i-1}\}$ have been constructed for some $1< i   \le s$, and satisfy $\Q(\zeta_{p^a},   \sqrt[p^{a}]{q_1},  \dots,  \sqrt[p^{a}]{q_{j-1}} ) \cap FL_{j-1}=\mb Q$ for any $a$ and every $1<j \le i$.
 Choose $q_i$ that  splits completely  in $\Q(\zeta_{p^{b_{i}}},   \sqrt[p^{b_i}]{q_1}, \dots,    \sqrt[p^{b_i}]{q_{i-1}})$ and is  inert in $ FL_{i-1}$. 
 Let $M_i/\Q$ denote the unique subextension of $\Q(\zeta_{q_i})/\Q$ of degree $p^{b_i}$. We now choose a cyclic subextension $L_i/\mb Q$ of $L_{i-1}M_i/\mb Q$  such that 
 \[
[L_i:\Q]=p^n, \quad L_i\cap M_i=\Q,
\quad [L_i\cap L_{i-1}:\Q]=p^{n-b_i}.
\]
Then $L_i$ still satisfies
\[
\Q(\zeta_{p^a},  \sqrt[p^{a}]{q_1}, \dots,   \sqrt[p^{a}]{q_i}) \cap FL_i=\Q
\]
for every~$a$.
 \end{itemize} 

From the construction above, we obtain the following properties.
\begin{enumerate}
 \item 
 For $1< i \le s$ and $1\le j \le i-1$,  the prime $q_j$ is  a $p^{b_i}$-th power modulo $q_i$. Hence, $q_j$ splits completely in $M_i$ by class field theory.  
  \item 
   We have  $L_{i}M_{i+1}=L_{i+1}M_{i+1}=L_iL_{i+1}$ for $1\le i\le r$,  and hence $L_iM_2M_3\cdots M_i=M_1M_2\cdots M_i $ for $1\le i \le s$.
  This results in 
 $$
 G_{M/L} \cong  \prod_{i=1}^{r}\mb Z/{p^{n_i}}\mb Z = \prod_{i=2}^{s}\mb Z/{p^{b_i}}\mb Z,
 $$
 where we put $L:=L_s$ and $M:=M_1\cdots M_s$. 
  
     \item 
  The primes $q_1, \dots, q_{i}$ do  not split in $L_i$. In particular $q_1, \dots, q_{s}$ do not split in $K=LF$. Moreover, for $1\le j\le i$, the ramification index and inertia degree of~$q_j$ in~$L_i$ are $p^{b_j}$ and $p^{n-b_j}$, respectively. For each subfield $N$ of~$MF$ such that $q_i$ does not split in~$N$, we denote by~$q_{i,N}$ the unique prime of~$N$ lying above~$q_i$.
   \item
   The extension $M/L$ is unramified, i.e.,  the $p$-Hilbert class field $H_L$ of $L$ contains $M$. Let $[q_{i, L}, M/L] $ be the Frobenius automorphism at $q_{i, L}$ in $G_{M/L}$. Then $[q_{1, L}, M/L]=0 $ and   $[q_{i, L}, M/L]$ is of order $p^{b_i}$ for $2\le i \le s$.
    
\end{enumerate}
\vskip 5pt
 
We claim that  $\{[q_{2, L}, M/L], \dots,   [q_{s, L}, M/L] \}$ generate $G_{M/L} $.
The case  $s=1$ is trivial since $M=L$ in this case.
For the case  $s> 1$, we note that $q_{j,L_{s-1}}$  splits completely in $L_{s-1}L/L_{s-1}$ for $1\le j\le s-1$   and that $q_{j,L }$  splits completely in $L_{s-1}L/L$ by (i)  and (iii).
Hence, from 
$$G_{M/L_{s-1}L} \cong G_{M_1\cdots M_{s-1}/L_{s-1}}$$
and induction hypothesis, it follows that $\{[q_{i, L}, M/L]\}_{2\le i \le s-1} $ generate $G_{M/L_{s-1}L} $.
Furthermore,    $q_{s,L }$ does not split and is unramified in $M_sL/L $, i.e.,   $[q_{s, L}, M_sL/L]$ generates $G_{M_sL/L}$.
Recalling  that  $M_sL=L_{s-1}L$  by (ii), we conclude that $\{[q_{i, L}, M/L]\}_{2\le i \le s} $ generate $G_{M/L} $.
Note that  $G_{M/L}=[q_{2, L}, M/L]^{\mb Z}  \cdots     [q_{s, L}, M/L]^{\mb Z}$ is a direct product by (iv).

\vskip 5pt

    Therefore, under the natural surjection
$
\phi:G_{H_L/L}\longrightarrow  G_{M/L} ,
$
we have
\begin{equation}\label{equation final}
\phi([q_{1, L}, H_L/L]) =0 \ \hbox{ and }\ \phi({X})= G_{M/L},
    \end{equation}
     where ${X}$ denotes the subgroup of $G_{H_L/L} $ generated by    $\{[{q}_{i,L}, H_L/L] \}_{2\le i \le s}$.

Let us denote $G_{K/F}\cong G_{L/\mb Q}$ by $P$. 
By  Lemma~\ref{lemma Iwasawa theorem} and Theorem \ref{theo-main}, the field $K$ has a local Minkowski unit at $p$ if and only if the kernel of the canonical map 
$$\Omega : (I^P_K/P_F )[p^\infty]\longrightarrow Cl_K^P[p^\infty]$$
is isomorphic to $\Z/p^n$. 

Consider the map $\Theta: (I^P_K/P_F )[p^\infty]\rightarrow G_{M/L}$, defined to be 
\begin{equation*}
\begin{tikzcd}
  \Theta: (I^P_K/P_F )[p^\infty] \arrow[r, "\Omega"]    & Cl_K^P[p^\infty] \arrow[r, "\bar{N}_{K/L}"]    & Cl_L^P[p^\infty] \arrow[d,   "\cong" sloped, "\psi" swap] &  \\
            &      &   (G_{H_L/L})^P\arrow[r, "\phi"]    &  (G_{M/L})^P = G_{M/L},
\end{tikzcd}    
\end{equation*}
where $\psi$ is the Artin map.
 We have  $(I^P_{K}/P_F) [p^\infty]\cong (I^P_K/I_F)[p^\infty] \cong \prod_{i=1}^{s}\mb Z/{p^{b_i}}\mb Z$ since $p\nmid h_F$.
Therefore, by (iii), $(I^P_{K}/P_F) [p^\infty]$ is generated by the classes of $q_{i,K}$ for $1 \leq i \leq r+1$.
From $N_{K/L} q_{i,K}= q^{[F:\Q]}_{i,L}$ and  \eqref{equation final}, it follows that $\Theta$ is surjective and $\ker \Theta \cong \mb Z/p^n \mb Z$.
Thus we get  $|\ker \Omega|\le  |\ker \Theta| =p^n$.

On the other hand,  we   have $|\ker \Omega | =|H^1(P, U_K)| \ge p^n$  by  Proposition~\ref{unit principal genus theorem}.
In particular, we obtain  $\ker \Omega =\ker \Theta \cong \Z/p^n$.
 Hence  $K$ admits a local Minkowski unit at $p$.

By  Proposition~\ref{unit principal genus theorem} again, we have $H^2(P, U_K)=0$.
Since the exact sequence  $1\rightarrow U_K \rightarrow K^\times \rightarrow P_K \rightarrow 1$ induces an injection $H^1(P, P_K)\hookrightarrow H^2(P, U_K)$, we get $\coker \ \Omega \cong  H^1(P, P_K)=0$ by \eqref{sequence genus theory}.
Hence it follows that  $\phi \mathbin{\scalebox{0.6}{$\circ$}}  \psi \mathbin{\scalebox{0.6}{$\circ$}} \bar{N}_{K/L}:Cl_K^P[p^\infty]\rightarrow G_{M/L}$ is an isomorphism.
Since all related groups are $p$-groups, we conclude that  $Cl_K[p^\infty]\cong Cl_L[p^\infty] \cong G_{M/L} $.
\end{proof}

\begin{rema}
In Theorem~\ref{theo-many ramification}, we can replace $\mb Q$ by an arbitrary number field $k$ under similar assumptions. More precisely, we assume that $F/k$ is cyclic and unramified at all infinite places, $\zeta_p\notin k$, $F\cap k(\zeta_p)=k$, and $p\nmid h_F\cdot[F:k]$. Then a similar proof shows that   there exist infinitely many cyclic extensions $L/k$ of degree $p^{n}$ such that: 
   \begin{itemize}
   \item
      $Cl_K[p^\infty] \cong Cl_L[p^\infty] \cong \prod_{i=1}^{r}\Z/p^{n_i}\Z$, where  $K=LF$.
      \item
       $K/k$ satisfies Isomorphism~\eqref{isoM}.
     \item
     $K/F$ is ramified at exactly $r+1$  primes.
           \end{itemize}            
To adapt the proof, we replace $q_i$, $\mb Q(\zeta_{q_i})$ and  $\Q(\zeta_{p^{b_{i}}},   \sqrt[p^{b_i}]{q_1}  \dots,    \sqrt[p^{b_i}]{q_{i-1}} )$ with $$\mf q_i,\  k(\mf q_i), \hbox{ and }\  k(\zeta_{p^{b_i}},    \sqrt[p^{b_i}]{U_k},    \sqrt[p^{b_i}]{x_1},  \dots,    \sqrt[p^{b_i}]{x_{i-1}} ),$$
respectively.
Here,  $\mf q_j$ is a prime ideal of $k$, $k(\mf q_j)$ denotes the ray class field of $k$ modulo $\mf q_j$, and  $x_j$ is a generator of $\mf q_j^{h_k}$.
\end{rema}

\begin{rema}
We are informed that Yahagi~\cite{Yahagi} constructed infinitely many cyclic extensions with prescribed $p$-class groups. 
However, his proof does not guarantee the existence of local Minkowski units in general.
\end{rema}

\begin{rema}\label{remark - ThmD-nonabelian}
Let $F=\Q$ and choose primes $q_1,\dots, q_s$ and a cyclic extension $L/\Q$ as in the proof of Theorem~\ref{theo-many ramification}.
After that, we choose a finite totally real Galois extension $F_0/\Q$ such that:
\begin{itemize}
    \item $p\nmid h_{F_0}\cdot [F_0:\Q]$.
    \item 
$q_1, \dots, q_s$ do not split in $F_0$.
   \end{itemize} 
Then a similar argument shows that $K=LF_0$ has a local Minkowski unit at $p$.   
\end{rema}

\begin{rema}
We remark that the $\Z[G]$-module structure of the unit groups of real biquadratic fields with many ramified primes was studied in \cite[Thm.~6]{MazurUllom} by exploiting explicit descriptions of systems of fundamental units. However, this explicit approach already becomes ineffective for bicubic fields.
\end{rema}

\section{Numerical examples}\label{sec-numexam}

We now present several examples of non-abelian extensions of number fields for which Isomorphism \eqref{isoM} holds.
All numerical computations were carried out using PARI/GP~\cite{PARI2}. We retain the notation $F/k$ and $K$ from Theorem~\ref{theo-smallramification}.

\subsection{Non-abelian examples with $K/F$ unramified}

The following group-theoretic proposition suggests the existence of many Galois extensions $K/k$ with $G_{K/k} \cong H \times \Z/p^n\Z$ admitting a local Minkowski unit at $p$ for various non-abelian groups $H$.

\begin{prop}
Let $k$ be a number field and $p$ a prime number. 
Let $F/k$ be a Galois extension of degree prime to $p$ such that $Cl_F[p^\infty]$ is cyclic.
Assume moreover that $\zeta_p \notin F$ and that $F/k$ is unramified at the infinite places. Let $K$ be the compositum of $F$ and the Hilbert $p$-class field of $k$. Then Isomorphism~\eqref{isoM} holds for $K/k$.
\end{prop}

\begin{proof}
By the cyclicity of  $Cl_F[p^\infty]$, $F$ and $K$ have the same $p$-Hilbert class field $H_F$. Let $\q$ be a prime of $F$ that is inert in $H_F$. Then $Cl_F[p^{\infty}]$ is generated by the class of $\q$, and $Cl_K[p^{\infty}]$ by the class of the prime of $K$ above $\q$. Hence, the natural map $Cl_F[p^{\infty}] \to Cl_K[p^{\infty}]$ is surjective, and $\CC_{K/F} \cong \Z/[K:F]\Z$. We therefore conclude from Theorem~\ref{theo-smallramification} (i) that the claim holds.
\end{proof}

\begin{exam}
Let $L$ be the $S_3$-number field defined by $x^6-6x^4+9x^2+104$. There are $73$ imaginary quadratic fields $\Q(\sqrt{-d})$ with $d\leq 650$ such that both $\Q(\sqrt{-d})$ and its compositum with $L$ have non-trivial cyclic $5$-class groups. These yield examples of $S_3\times\Z/5^s\Z$-extensions admitting a local Minkowski unit at $5$ for some $s\geq 1$.
\end{exam}

\subsection{Local Minkowski units with Galois group $S_3 \times \Z/p^n\Z$ in the ramified case}

We provide examples for Theorem~\ref{theo-smallramification} with $H=S_3$. We first construct an $S_3$-extension as follows. Suppose that $q_1, \ldots, q_s$ are pairwise distinct primes satisfying $q_i \equiv 2 \pmod{3}$, and consider a polynomial
\begin{equation*}
 f(X)=X^3-q_1\cdots q_s aX+q_1\cdots q_s \in \Z[X]
\end{equation*}
for some integer $a>1$. Since $f(X)$ is Eisenstein at each $q_i$, for any root $\alpha$ of $f(X)$, the extension $\Q_{q_i}(\alpha)/\Q_{q_i}$ is totally ramified of degree $3$. Since $\zeta_3 \notin \Q_{q_i}$, the extension $\Q_{q_i}(\alpha)/\Q_{q_i}$ is not Galois. Hence, the Galois group of the splitting field of $f(X)$ over $\Q_{q_i}$ is isomorphic to $S_3$. Moreover, since
\begin{equation*}
f(0)=q_1\cdots q_s>0
\quad\text{and}\quad
f(1)=1-q_1\cdots q_s a+q_1\cdots q_s<0,    
\end{equation*}
 $f(X)$ has three real roots. Therefore, the splitting field $F$ of $f(X)$ over $\Q$ is totally real. In particular, we have $G_{F/\Q}\cong S_3$, and each $q_i$ has a unique prime above it in $F$. If there exists a cyclic extension $L/\Q$ of $p$-power degree such that $L$ and the primes $q_i$ satisfy the assumptions of Theorem~\ref{theo-many ramification}, and if $p\nmid h_F$, then the compositum $FL$ provides an example as in Remark~\ref{remark - ThmD-nonabelian}. We also give the following numerical example.

\begin{exam}[Examples for Theorem \ref{theo-smallramification} and Theorem \ref{theo-Iwasawa}]
If we take $s=1$ and $a=2$, then PARI/GP computations show that the $p$-class group of the $S_3$-extension defined by $X^3-2pX+p$ is trivial for all odd primes $p<1000$ with $p \equiv 2\pmod{3}$. Hence, by taking the compositum of each such extension with the cyclotomic $\Z_{p}$-extension of $\Q$, we obtain number fields with Galois group $S_3\times \Z/p^n\Z$ admitting a local Minkowski unit at $p$ for all $n\in\N$.
\end{exam}

\subsection{Local Minkowski units in $D_{2q} \times \Z/p^n\Z$-extensions with large $q$}

For a number field $L$, the governing field $\mathrm{Gov}(L)$ for $p$ is defined by 
$$\mathrm{Gov}(L):=L(\zeta_p, \sqrt[p]{V_{\emptyset}(L)}),$$
where $V_{\emptyset}(L)$ denotes the multiplicative group of elements $\alpha \in L^{\times}$ such that the principal ideal $(\alpha)$ is the $p$-th power of a fractional ideal of $L$. If $p\nmid h_L$, then $\mathrm{Gov}(L)=L(\zeta_p, \sqrt[p]{U_L})$ (cf. \cite{GM-PMB, HMR24-GM}). The governing field $\mathrm{Gov}(L)$ provides useful criteria for the existence of cyclic extensions of $L$ of degree $p$ that are ramified at a prescribed set of non-$p$-adic primes.

\begin{exam}\label{exam-GM}
Let $q$ be an odd prime, and let $k \neq \Q(\sqrt{q})$ be a real quadratic field with $q\nmid h_k$.
Let $\mathrm{Gov}(k)=k(\zeta_q, \sqrt[q]{U_k})$ be the governing field of $k$ for $q$.
By the Schur-Zassenhaus theorem, we have
\begin{equation*}
    G_{\mathrm{Gov}(k)/\Q} \cong G_{\mathrm{Gov}(k)/k(\zeta_q)} \rtimes \big (G_{k/\Q} \times G_{\Q(\zeta_q)/\Q} \big ).
\end{equation*}
By the Chebotarev density theorem, there exists a prime $\ell$ such that
\begin{equation*}
       [\ell,  \mathrm{Gov}(k)/\Q ] = (id_{ \mathrm{Gov}(k)}, (\tau_1, \tau _{2})) \in G_{\mathrm{Gov}(k)/k(\zeta_q)} \rtimes \big ( G_{k/\Q} \times G_{\Q(\zeta_q)/\Q} \big ).
\end{equation*} 
Here, $\tau_1$ and $\tau_2$ denote elements of order $2$ in  $G_{k/\Q}$ and $G_{\Q(\zeta_q)/\Q}$, respectively.
More precisely,  we fix a prime $\mf L $ of $\mathrm{Gov}(k)$ above $\ell$ such that its Frobenius automorphism $Fr_{\mf L}$ is $ (id_{ \mathrm{Gov}(k)}, (\tau_1, \tau _{2}))$, and we choose $[\ell,  \mathrm{Gov}(k)/\Q ]$ to be $Fr_{\mf L}$ although $ [\ell,  \mathrm{Gov}(k)/\Q ]$ is defined only up to conjugation.
From
$[\ell, k(\zeta_q)/\Q]=(\tau_1,\tau_2)$ we have $\ell \equiv -1 \pmod{q}$ and $\ell$ does not split in $k$. Let $\mathfrak{l}$ be the prime of $k(\zeta_q)$ below $\mf L$. Then we have
\begin{equation*}
     [\mathfrak{l},  \mathrm{Gov}(k)/k(\zeta_q) ] = [\ell,  \mathrm{Gov}(k)/\Q]^2 = (id_{ \mathrm{Gov}(k)}, (\tau_1, \tau_2))^2 =1. 
\end{equation*}
Hence, by a theorem of Gras and Munnier (cf. \cite{GM-PMB, HMR24-GM}), there exists a cyclic extension $k(\mathfrak{l})'/k$ of degree $q$ that is precisely ramified at the unique prime of $k$ above $\ell$. In particular, $\ell$ does not split in $k(\mathfrak{l})'$. The group $G_{k(\mathfrak{l})'/\Q}$ is non-abelian since $\ell \equiv -1 \pmod{q}$. Hence, we have $G_{k(\mathfrak{l})'/\Q} \cong D_{2q}$.

For example, take $k=\Q(\sqrt{3})$, $q=7$, and $\ell=1049$. Then $k$ admits a $\Z/7\Z$-extension $k'$ ramified at the unique $1049$-adic prime. A numerical computation shows that $k'$ has class number $1$. Therefore, for every $n\geq 1$, the compositum of $k'$ with the $n$th layer of the cyclotomic $\Z_p$-extension of $\Q$, for $p=1049$, yields a $D_{14}\times\Z/1049^n\Z$-extension admitting a local Minkowski unit at $1049$. 
\end{exam}

\bibliographystyle{plain}
\bibliography{references}

\end{document}